\documentclass[12pt,letterpaper]{amsart}

\usepackage[pdftex]{hyperref}
\usepackage{color,graphicx,array,amscd,amsmath,amssymb,abbrViro,wrapfig}
\newcommand\sfit{\em\sffamily}

\graphicspath{  {./figs/} }
\begin{document} 

\title[Biflippers and head to tail]{Biflippers and head to tail composition rules}
\author{Oleg Viro}
\dedicatory{Stony Brook University, NY, USA} 
\address{Department of Mathematics, Stony Brook University,
Stony Brook NY, 11794-3651, USA
}   
\email{oleg.viro@gmail.com}

\begin{abstract} 
A new graphical calculus for operating with isometries of 
low dimensional spaces of classical geometries is proposed. 
It is similar to a well-known graphical
representation  for vectors and translations in an affine space. 
Instead of arrows, we use biflippers, which are arrows framed  at the 
end points with subspaces. 
The head to tail addition of vectors and composition of translations is 
generalized to head to tail composition rules for isometries.     
\end{abstract}
\maketitle

\section*{Introduction}\label{s0}

Any isometry of Euclidean space is a composition of
two symmetries of order two. A symmetry of order two is determined by its 
set of fixed point. Drawing of the fixed point sets gives an easy and
effective graphical way to present isometries and operate with them.
Below a new technique for this is proposed. 

It is naturally limited to low dimensions (up to 3), 
because in the higher
dimensions the situation becomes more complicated on the one hand and 
drawing less practical, on the other hand. In the low dimensions, the same
technique works for other classical geometries: elliptic
(both in spheres and projective spaces), hyperbolic, and conformal.  
\medskip

\subsection*{\sffamily Added in August 2015}\label{s0a1}
{\sffamily
After uploading the original version of this preprint in December
2014 to the arXiv, the
author has come across several publications in which some 
results of this preprint had been obtained. Below, the appropriate references
and discussions are inserted in paragraphs similar to this one. They all 
are marked by words ``Added in August 2015''. 

Although the goal of the present paper (creation of graphical tools for 
operating with isometries in low-dimensional spaces) was not pursued in
these publications, they put the present paper in a broader context,
and reduced its novelty. 
}

\tableofcontents

\section{Portraits of 
isometries}\label{s1}

\subsection{Flips and flippers.}\label{s1.1}
Let $X$ be an Euclidean space (say, a plane or the 3-space) or a subset of
a Euclidean space. 
Let $A$ be a non-empty subset of $X$. If 
\begin{itemize}
\item  there exists an 
isometry $F:X\to X$\\ (i.e., a map which preserves distances between
points),
\item which is an involution\\ (i.e., $F\circ F=\id$),
\item $A$ is the fixed point
set of $F$\\ (i.e., the set of all points of $X$ that are not moved by $F$), 
\item and if such $F$ is unique,
\end{itemize}
then $A$ is called a {\sfit flipper\/} in $X$, 
and $F$ is called a {\sfit flip\/} in $A$ and is denoted by $F_A$.  

 Flipper $A$ and flip $F_A$ determine each other. In particular, 
$A$ encodes $F_A$.
In low dimensional cases, if $A$ is clearly drawn, the picture of
$A$ is a portrait for $F_A$. 
\smallskip

\noindent
{\bf Examples.}
Obviously, in any $X$ the whole $X$ is a flipper, the corresponding 
flip is the identity map.

On the plane, a set is a flipper iff it is either a line or a point.
 More generally, in a Euclidean space of any dimension, a set is a
flipper iff it is a non-empty affine subspace. In a Euclidean space of any
dimension, if an isometry is involution, then it is a flip in a
non-empty affine subspace.

On the 2-sphere there are two kinds of flippers: great circles and pairs of
antipodal points. The corresponding flips are the restrictions of the
flips of the ambient 3-space with the flippers passing through the
origin.\smallskip

\noindent{\bf Generalizations.}
The notions introduced above can be considered in much more general setup:
$X$ may be a whatever space with whatever structure, and the r\^ole of 
isometries would play automorphisms of this structure. 
In the last section we will consider the 2-sphere with conformal structure. 
\smallskip 

\noindent
{\bf About uniqueness. }
In all the examples above uniqueness of $F$ was not a restriction: if $F$
existed, then it was unique. The uniqueness condition was included in order
to ensure one-to-one correspondence between flips and flippers.
There are sets in more complicated
figures, which are fixed point sets of several different 
involutions-isometries.

 One of the simplest examples: $X$ is the union of two
perpendicular lines $L$ and $M$ on the plane, and $A$ is the 
intersection point of the lines. Then $A$ is the fixed point set for three
involutions-isometries. They are restrictions to $X$ of three flips of
the ambient plane: the central symmetry in $A$ and the 
reflections in the two lines bisecting the angles between $L$ and $M$.
Therefore $A$ is not qualified to be a flipper of $X$, although $A$ is a
flipper of the ambient plane.  
\smallskip

\noindent{\bf Terminological remark.}
The words {\sfit flip\/} and {\sfit flipper\/} seems to be new in this context.
Old words that are used in mathematics with a close meanings are 
{\it symmetry, reflection\/} and {\it inversion.} 

The closest is {\it reflection\/}, but it has a well-established meaning.  
Most mathematicians believe that the fixed point set of a reflection must 
have codimension one. This kind of reflections plays a special role in 
geometry and group theory, but is too narrow for our purposes. 
A use of the word reflection in the sense above (i.e., for what we call a 
flip) would be considered a language abuse, see Wikipedia article {\it Point
reflection.\/} 

The word {\it inversion\/} is used for several different notions including 
the symmetry with respect to a point and the symmetry with respect to a 
sphere. The latter seems to be commonly used and loading yet another
meaning would be fairly criticized. 

The word {\it symmetry\/}, is commonly used in much broader sense than we 
need, often just as a synonym for automorphism.

\subsection{Flip-flop decompositions.}\label{s1.2}
A presentation of an isometry $T$ as a composition of two 
flips $F_B\circ F_A$ is called a {\sfit flip-flop decomposition\/} of
$T$. An ordered pair $(A,B)$ of flippers $A$ and $B$  is called a {\sfit
biflipper}. 
The biflipper $(A,B)$ encodes the flip-flop decomposition 
$T=F_B\circ F_A$ and hence the isometry $T$.

If $A$ is a flipper in a space $X$, then the flip $F_A$ has a flip-flop
decompositions $F_A\circ F_X$ and $F_X\circ F_A$, because $X$ is a flipper
in itself and $F_X=\id$. 
However, often a flip has other flip-flop
decompositions. Flip-flop decompositions $F_A=F_X\circ F_A=F_A\circ F_X$ are
said to be {\sfit improper}, all other flip-flop decompositions are said to be
{\sfit proper}. A biflipper is said to be {\sfit proper\/}, if none of its two
flippers is the whole space.

A flip-flop decomposition of an isometry gives an opportunity to draw a
picture, which completely describes the isometry. 
However, a picture showing just $A$ and $B$ would be incomplete. It does not
show a crucial bit of information: which of the flipper is the first in the
biflipper and which is the second. In other words, it does not show 
the order of $F_A$ and $F_B$  in the flip-flop decomposition. The order
distinguishes an isometry from the inverse isometry. Indeed, since
flips are involutions, $F_A\circ F_B=F_A^{-1}\circ F_B^{-1}=(F_B\circ
F_A)^{-1}$.

{\sffamily
\subsection*{\sffamily Added in August 2015}\label{s1a1}
{\sfit Existence of a flip-flop decomposition is known as 
strong reversibility.\/} A presentation
of a group element as a product of two elements of order two was studied
and proved to be useful in
many contexts. See a recent survey \cite{O'Farrell}. Group elements that admit
such representation are called {\sfit strongly reversible.\/} This name
comes from the relation to the following notion of {\sfit reversibility:\/} 
a group element is said to be {\sfit reversible\/} if it is conjugate (in the
group) to its own inverse. In formulas: $f\in G$ is reversible if there
exists $h\in G$ such that $f^{-1}=h^{-1}fh$. Strong reversibility implies 
reversibility.
Indeed, if $f=\Ga\Gb$ and $\Ga^2=\Gb^2=1$, then
$f^{-1}=\Gb^{-1}\Ga^{-1}=\Gb\Ga=\Ga^2\Gb\Ga=\Ga(\Ga\Gb)\Ga=\Ga^{-1}f\Ga$.  

Strong reversibility of isometries of the classical homogeneous spaces
(Euclidean, hyperbolic, and elliptic) was studied by Short \cite{Short}
and, independently, by Basmajian and Maskit \cite{BM}.

\subsection*{A topological view on reversibility and strong 
rever\-si\-bi\-lity}\label{s1a3}
Obviously, an element $x$ of a group $G$ is reversible if and only if 
there exists a homomorphism $F:K\to G$ of the group 
$K=\{m,l\mid mlml^{-1}=1\}$ such that $F(m)=x$. 

The group $K$ is the
fundamental group of the Klein bottle. Its element $m$ is realized by 
a meridian (a simple closed curve cutting along which turns the Klein bottle
to the cylinder). An element $x$ of the fundamental group of a topological
space $X$ is reversible if and only if there is a continuous map of the
Klein bottle to $X$ which maps a meridian to a loop realizing $x$. 
 
Similarly, $x\in G$ is
strongly reversible if and only if there are two homomorphisms 
$\mathbb Z/_2\to G$ such that $x$ is the product of the images 
of the non-trivial element. 

The relation between reversibility and strong reversibility has a
topological counterpart. The Klein bottle is a connected sum 
of two copies of the projective plane. Contracting the neck of the 
connected sum gives a
continuous map of the Klein bottle onto the bouquet of two projective
planes, under which a meridian of the Klein bottle is mapped onto a 
bouquet of the projective lines. 
\subsection*{A remarkarble reformulation for strong reversiblity}\label{s1a4}
{\em An element $x$ of a group $G$
is strongly reversible if and only if there exists an involution $\Ga\in G$
such that $\Ga x$ is an involution.\/} 

Indeed, if $\Ga$ and $\Ga x=\Gb$ are involutions, then $x=\Ga^2 x=\Ga\Gb$ 
is a product of two involutions, hence strongly reversible. Conversely, 
if $x=\Ga\Gb$ where $\Ga$ and $\Gb$ are involutions, 
then $\Ga x=\Ga^2\Gb=\Gb$ is an involution.\qed
}

\subsection{What an arrow may be for.}\label{s1.3}
Often, an arrow in a mathematical picture portrays nothing but an ordered 
pair of points, the arrow tail and arrowhead. 
Individual points are difficult to discern on a picture,
and, in order to make the points more visible, they are connected with a
line segment, and, in order to show which point is first and which is the
second, the segment is turned into an arrow directed
from the first of the points to the second.

\subsection{Arrows between flippers.}\label{s1.4}
\begin{wrapfigure}[5]{r}{1.8in} 
{\includegraphics{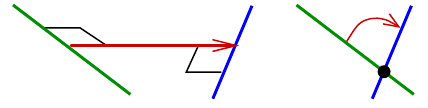}}
\end{wrapfigure}

We need to portray an ordered pair of flippers. 
The flippers may happen to be points in an Euclidean space, and then 
we just draw 
the arrow between them, as usual. In a general situation, we draw the
flippers 
and connect them with an arrow, in order to show the ordering. 
If the flippers are disjoint affine subspaces in a Euclidean space, 
then we choose an arrow along a common 
perpendicular to the both subspaces. If the flippers
intersect, we still need to show their ordering, and, to this end, we connect
them with an arc-arrow.

In order to distinguish the first and the second flippers in words, we will
call the first flipper in a biflipper the {\sfit tail\/} flipper and the
second one, {\sfit the head\/} flipper, according to the place occupied by them in a
picture: at the arrow tail or arrow head, respectively.

A pair of flippers $A$ and $B$ together with an arrow, which connects $A$ to 
$B$ and specifies the ordering, is called a {\sfit arrowed biflipper\/} (or
just a {\sfit biflipper\/} or just an {\sfit arrow\/}, when there is no danger 
of confusion) and denoted by  
$\overrightarrow{AB}$. Thus an arrow $\overrightarrow{AB}$ provides a 
graphical encoding for the
flip-flop decomposition  $T=F_B\circ F_A$ and, in particular, a graphical
encoding, a portrait for $T$.

\subsection{Biflippers on the plane.}\label{s1.5}

It is easy to list all kinds of biflippers that exist on the 
plane. Take all possible mutual positions of two different flippers 
(recall that a proper flipper on the plane is either a points or a line),
and connect the flippers with an arrow.
Proper biflippers look as follows:\\ 

\centerline{\includegraphics[scale=.9]{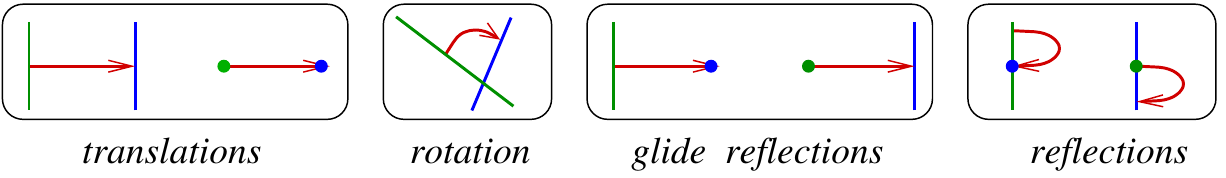}}
\vspace{2pt}\noindent
These biflippers encode all the plane isometries (except the identity).
Let us consider them one by one. 

\subsection{Translations.}\label{s1.6}
biflippers \ $\fig{frAr-tr-1.pdf}$ \ and \ $\fig{frAr-tr-2}$ \ encode the translation in the 
direction of the arrow by the distance {\bfit twice the length\/} of the arrow.
Here are sketches of the well-known proofs: \ 
$\fig{frAr-tr-1-pf.pdf}$ \ and \ $\fig{frAr-tr-2-pf.pdf}$ 

\subsection{Rotations.}\label{s1.7}
Biflipper \ $\fig{frAr-rot.pdf}$ \ encodes the rotation about 
the intersection 
point of the flippers, in the direction from the tail flipper to the head 
flipper, 
by the angle twice greater than the angle between them. A sketch of
proof: $\fig{frAr-rot-pf.pdf}$.

{\bfit
Warning:} often 
a rotation is shown by a similar picture, but with 
the angle between the lines twice greater and the direction of 
the arc-arrow indicating if the rotation is clockwise or counter-clockwise. 
Recall that in a biflipper the arc-arrow points from the tail flipper to
head flipper. The following two biflippers and a traditional picture 
show rotations in the same clockwise direction by $151^\circ$: \\ 
$\fig{frAr-rot.pdf}$ \ \ $\fig{frAr-rot2.pdf}$ \ \ $\fig{pic-rot}$
 
\subsection{Glide reflections.}\label{s1.8}
Biflippers \ $\fig{frAr-gl-1.pdf}$ \  and \ $\fig{frAr-gl-2.pdf}$ \ encode  glide reflections. 
This is the 
reflection in the line which contains the arrow, followed by the translation 
in the direction of the arrow by the distance twice greater than the length
of the arrow. Sketches of proofs:
$\fig{frAr-gl-pf-1.pdf}$ \ \ and \ \ $\fig{frAr-gl-pf-2.pdf}$ 

\subsection{Reflections.}\label{s1.9}
Biflippers \ $\fig{frAr-refl-1.pdf}$ \ and \ $\fig{frAr-refl-2.pdf}$ \ encode reflections;
specifically, the reflections in the line which is  perpendicular
to the line-flipper, and erected from the point-flipper. Sketches of proofs:  
 $\fig{frAr-refl-1-pf.pdf}$ and $\fig{frAr-refl-2-pf.pdf}$
These are degenerate versions of the biflippers that encode glide
reflections, in the same sense as reflections in
lines are degenerate glide reflections. 

\subsection{The identity.}\label{s1.10}
The identity is also an isometry and it is encoded by any
biflipper consisting of two identical flippers. The identity is a 
special case (i.e., a degeneration) for translations as well as for rotations. 

\section{Biflippers of the same isometry}\label{s2}

\subsection{Equivalence of biflippers.}\label{s2.1}
Biflippers $\overrightarrow{AB}$ and $\overrightarrow{CD}$ are said to
be {\sfit equivalent\/} \ if $F_B\circ F_A=F_D\circ F_C$.

Below we describe two easy ways to change a biflipper 
 into a different, but equivalent biflipper.

\begin{thm}\label{th01}
Let $A$, $B$ and $C$ be flippers such that $F_C$ commutes both
with $F_A$ and $F_B$. Then  $F_A\circ F_C$ and $F_B\circ F_C$
are involutions. If these involutions are flips in flippers $A'$ and $B'$ 
(i.e.,  $F_A\circ F_C=F_{A'}$ and $F_B\circ F_C=F_{B'}$), 
then $\overrightarrow{AB}$ is equivalent to $\overrightarrow{A'B'}$.
\end{thm}

\begin{proof} If two involutions, say $\Gs$ and $\Gt$, commute, then their 
product is an involution. Indeed, by multiplying the equality $\Gs\Gt=\Gt\Gs$ by $\Gs\Gt$
from the right, we obtain $(\Gs\Gt)^2=\Gt\Gs\Gs\Gt=1$.  Further, 
$$
F_{B'}\circ F_{A'}= F_B\circ F_C\circ F_A\circ F_C=
F_B\circ F_C\circ F_C\circ F_A=F_B\circ F_A.
$$
Thus $\overrightarrow{AB}$ is equivalent to $\overrightarrow{A'B'}$. 
\end{proof}

In a Euclidean space any isometry of order two is a reflection in a
subspace. Thus for Euclidean space the second assumption of 
Theorem \ref{th01} is nothing but 
just a definition for $A'$ and $B'$.\medskip

\noindent
{\bfit Examples. } Let $A$, $B$ and $C$ be lines on the plane,  $A\perp
C\perp B$, so $A\parallel B$. Let $A'=A\cap C$ and $B'=B\cap C$. Then  
\begin{enumerate} 
\item $\overrightarrow{AB}$ and $\overrightarrow{A'B'}$ are equivalent;
\item   $\overrightarrow{AB'}$ and $\overrightarrow{A'B}$ 
are equivalent.
\end{enumerate} 

As we saw above, in the first example the arrows encode the 
same translation, in the second - the same glide reflection.
 
\begin{thm}\label{th02}
Let $S$ be the isometry 
encoded by $\overrightarrow{AB}$ and let $T$ belong to the centralizer of
$S$ in the isometry group (i.e., $T$ be 
an isometry, which commutes with $S$). Then 
$\overrightarrow{T(A)T(B)}$ is equivalent to  $\overrightarrow{AB}$.
\end{thm}
\begin{proof} Notice, that if $C$ is a flipper and $T$ is any isometry, then
$T(C)$ is a flipper, and
 $F_{T(C)}=T\circ F_C\circ T^{-1}$. Therefore,  
\begin{multline*}
F_{T(B)}\circ F_{T(A)}=
T\circ F_B\circ T^{-1}\circ T\circ F_A\circ T^{-1}\\=
T\circ F_B\circ F_A\circ T^{-1}=
TST^{-1}=S.
\end{multline*}
\end{proof}

\subsection{Equivalence of plane biflippers.}\label{s2.2}
Recall that biflippers, which encode the same isometry, are said to be 
equivalent. For different biflippers equivalence may look quite 
different. 

Arrowed biflippers, that encode translations, are equivalent iff the 
underlying arrows (i.e., the bare arrow, with the flippers removed), can be 
obtained from each other by a translation. 

Arrowed biflippers, that define glide reflections, are
equivalent iff the underlying arrows lie on the same line and have the same
length and direction. In other words, they can be obtained from each other
by a translation along this line. The equivalence classes of the underlying
arrows are known as {\bfit sliding vectors.\/} By Theorem \ref{th01}, it 
does not matter whether the tail flipper or the head flopper is a line. 
Simultaneous change of the dimensions of the tail and head flippers  does not
change the gliding reflection. 

Arrowed biflippers, that define rotations, are equivalent 
iff they can be obtained
from each other by a rotation about the intersection points of
the flippers (whis is the center of the rotation). 

Arrowed biflippers, that define reflections, are equivalent iff they can be obtained
from each other by a translation in the direction perpendicular to the
line-flipper (this is the direction of the reflection mirror).

\section{Head to tail composition method}\label{s3}

\subsection{Generalities.}\label{s3.0}
If  $S=I\circ J$, $T=J\circ K$, and $J$ is an involution,
then, obviously, $S\circ T=I\circ J^2\circ K=I\circ K$. 

Therefore if isometries $T$ and $S$ are encoded by biflippers 
$\overrightarrow{AB}$ and  $\overrightarrow{BC}$, respectively, then 
 $\overrightarrow{AC}$ encodes $S\circ T$. Here it is important that 
the head flipper of the first biflipper coincides with the tail flipper 
of the second. 

If isometries $T$ and $S$ can be represented by such biflippers, we say
that they {\sfit admit head to tail composition.\/} An algorithm for finding 
biflippers of this kind is called a {\sfit head to tail composition method\/}. 

The head to tail composition methods depend on the types of the composed 
isometries. It's remarkable that a head to tail composition 
method exists for any pair of isometries of the Euclidean plane and in many
other similar setups. However, as we show below, for some pairs of isometries
of the hyperbolic plane and Euclidean 3-space it does not exist. 

The problem of finding a biflipper, which encodes a composition of
isometries, given biflippers, that encode these isometries, is solved 
easily in those two situations by slightly more complicated methods. 
See sections \ref{s5} and \ref{s7} below. 

\subsection*{\sffamily Added in August 2015}\label{s3a5}
{\sffamily
{\em Isometries that admit tail to head composition are known as linked.} 
The term {\em linked} was introduced  by Basmajian and 
Maskit \cite{BM}, although the concept appeared before in the context of
low dimensional hyperbolic geometry. 
 Strongly invertible isometries $T$ and $S$ are called {\em linked\/} 
if there exist their presentations as compositions of involutions with 
the same involution involved in presentations of both $T$ and $S$. 

Observe that if $T=\Ga\Gb$ is such a
presentation, then there exist involutions $\Gg$ and $\Gd$ such that
$T=\Gg\Ga$ and $T=\Gb\Gd$. Namely, $\Gg=T\Ga^{-1}=\Ga\Gb\Ga^{-1}$ and
$\Gd=\Gb^{-1}T=\Gb^{-1}\Ga\Gb$, which are involutions as elements conjugated
to involutions. Cf. \cite{Silverio}.

 Hence if $T$ and $S$ are linked and $\Gb$ is an involution involved in
presentation of each of them as products of two involutions, then 
$T$ and $S$ admit  
presentations as products of involutions $T=\Ga\Gb$ and $S=\Gb\Gg$ 
and hence admit head to tail composition. 

Basmajian and Maskit \cite{BM} proved that, in dimensions $>3$, 
almost every pair of orientation preserving isometries of a classical space 
form (Euclidean,
elliptic or hyperbolic) is not linked. They proved that this is the
case also for the Euclidean 3-space. However the proof contains a mistake 
and, in fact, as is shown below in Section \ref{s7.6}, the statement is 
not true: any two orientation
preserving isometries of the Euclidean 3-space are linked.

Beyond the general underlying 
notions, this paper have almost no intersections with \cite{BM}.
The author was interested mainly in low-dimensional cases, pictures 
and algorithmic aspects.  Basmajian and Maskit  
\cite{BM} deal mainly with dimensions $>3$ and include no picture. 

Basmajian and Maskit were mostly interested in orientation
preserving isometries and their presentations as compositions of orientation
preserving involutions. Generic orientation reversing isometries appear in
\cite{BM} only once (in Corollary 1.3), orientation reversing involutions
appear only as long as they appear in decompositions of orientation
preserving isometries. In the present paper orientation reversing
isometires are not discriminated.}   

\subsection{Translations.}\label{s3.1} 
Choose biflippers, which encode the
translations, with 0-dimensional 
flippers 
and move one of them by an appropriate translation so that  
the head flipper of the first biflipper would coincide 
with the tail flipper of the second
biflipper. After that this looks as the usual head to tail addition of vectors.\\
\centerline{\includegraphics{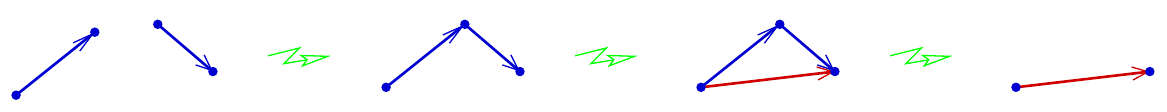}}
\subsection{Rotations.}\label{s3.2} Take any biflippers encoding the
rotations. Turn the biflippers around the centers in order to make the
head flipper in the first of them coinciding with the tail flipper in the
second. If after this the other two lines are not parallel, then
draw an arc-arrow connecting the tail flipper of the first biflipper
to the head flipper of the second biflipper. Erase the coinciding
lines and the old arc-arrows. The composition is a rotation. 
Notice that this head to tail construction gives the center of the rotation  
composition.\\
\centerline{\includegraphics{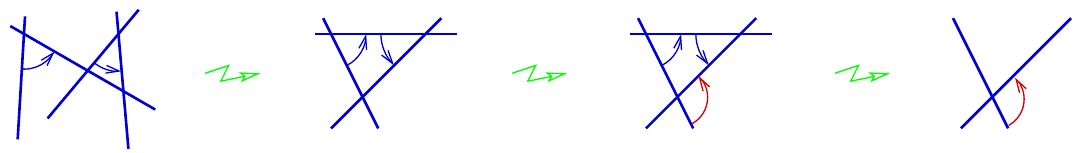}}
If the tail flipper of the first biflipper is parallel to the head flipper
of the second biflipper, then after erasing everything besides these
two parallel lines and connecting them with an arrow perpendicular to them, 
we obtain a
biflipper for a translation.\\ 
\centerline{\includegraphics{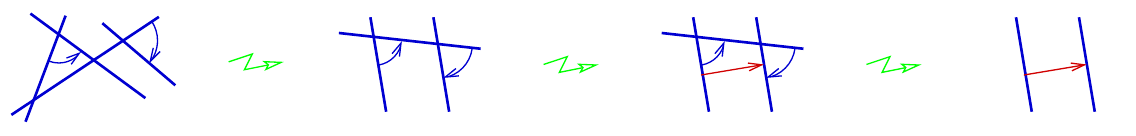}}
\subsection{Translation followed by rotation.}\label{s3.3} Choose a biflipper 
representing the 
translation to be formed by lines. Turn a biflipper representing
the rotation to make the flipper perpendicular to the
direction of the translation (i.e., parallel to the lines which form
 the biflipper representing the translation). By a
translation of the biflipper representing the translation, superimpose 
the appropriate lines.\\
\centerline{\includegraphics{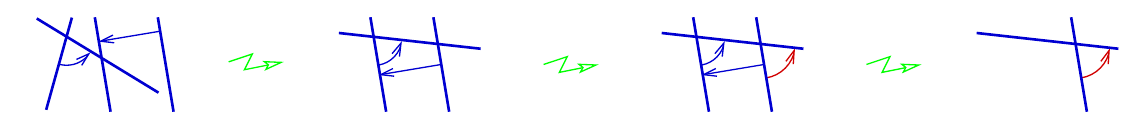}}
\subsection{Translation followed by glide reflection.}\label{s3.4}
The composition is a glide reflection. Here is how to find its
biflippers. \\    
\centerline{\includegraphics{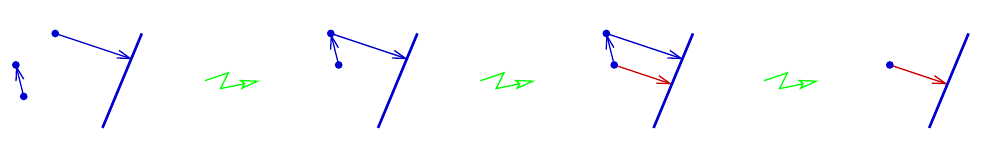}}
\subsection{Two glide reflections with non-parallel axes.}\label{s3.5}
 The composition 
is a rotation. Here is how to find its biflipper.\\
\subsection{Two glide reflections with parallel axes.}\label{s3.5'}
\centerline{\includegraphics{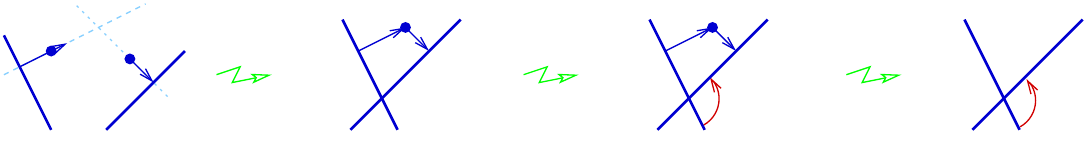}}
The composition is a translation. Here is how to find its biflipper. \\
\centerline{\includegraphics{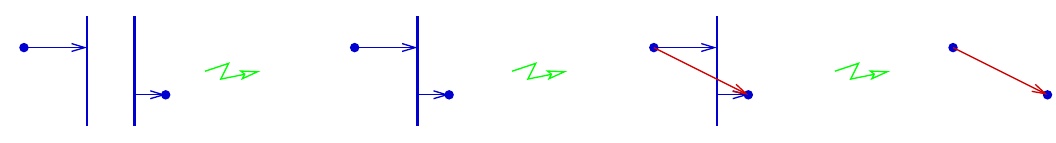}}

\subsection{Generators and relations for the plane isometry group}\label{s3.6}
Non\-uniqueness of flip-flop decompositions of a translation implies that
the composition $F_l\circ F_m$ of reflections in parallel lines $l$ and $m$
does not change if one
replaces $l$ and $m$  by their images  $l'$ and $m'$ under any translation.
Thus we have a relation $F_l\circ F_m=F_{l'}\circ F_{m'}$.

Similar relations follow from non-uniqueness of flip-flop decompositions of
a rotation. Namely, if lines $l$ and $m$ intersect at a point $A$ and lines
$l'$ and $m'$ are their images under some rotation about $A$, then  
$F_l\circ F_m$ and $F_{l'}\circ F_{m'}$ are flip-flop decompositions of the
same rotation, and hence   $F_l\circ F_m=F_{l'}\circ F_{m'}$.

Observe that in both situations the lines $l$, $m$, $l'$ and $m'$ belong to
a pencil: in the first situation they are all parallel to each other, in
the second, they have common point $A$. We will call them {\sfit pencil 
relations\/}. Besides, any flip is an involution. Thus $F_l^2=\id$.

\begin{thm}[O.Viro \cite{Viro}]\label{thRelCompl}
 Any relation among reflections in the group of isometries of the 
Euclidean plane is a corollary of pencil and involution relations.
\end{thm}

\begin{lem}\label{lem1}
Any composition of four reflections can be converted 
by pencil and involution relations into a composition of two reflections.
\end{lem}

\begin{proof}
Consider a composition  $F_n\circ F_m\circ F_l\circ F_k$ of four reflections.
If any two consecutive lines coincide (i.e., $k=l$, or $l=m$, or $m=n$), 
then, by applying the involution relation, we can eliminate the
corresponding reflection from the composition. So, in the rest of the
proof, we assume that none of consecutive lines coincide.

Compositions $F_n\circ F_m$ and $F_l\circ F_k$ are flip-flop decompositions
of translations or rotations. If at least one of them is a rotation, we can
apply head to tail rule for identification of their composition. Notice
that all the transformations used in this method are applications of the
pencil and involution relations.

Let both $F_n\circ F_m$ and $F_l\circ F_k$ be translations. Then
$k\parallel l$ and  $m\parallel n$. If  $l\cap m\ne\empt$, then by rotating 
the middle pair of lines $l,m$ by right angle we replace them by lines $l'$
and $m'$ such that $k\perp l'$ and $m'\perp n$ and obtain
the situation of two rotations for which head to tail method works.\\ 
\centerline{$\vcenter{\hbox{\begin{picture}(0,0)%
\includegraphics{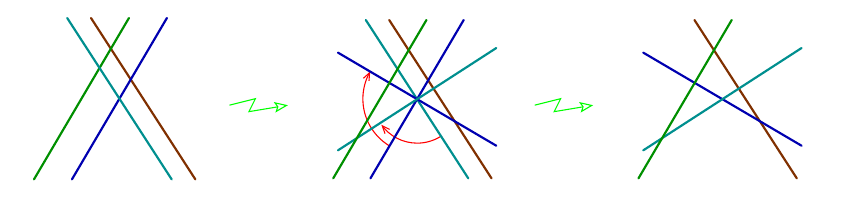}%
\end{picture}%
\setlength{\unitlength}{2735sp}%
\begingroup\makeatletter\ifx\SetFigFont\undefined%
\gdef\SetFigFont#1#2#3#4#5{%
  \reset@font\fontsize{#1}{#2pt}%
  \fontfamily{#3}\fontseries{#4}\fontshape{#5}%
  \selectfont}%
\fi\endgroup%
\begin{picture}(5866,1532)(-14,-602)
\put(4186,-496){\makebox(0,0)[lb]{\smash{{\SetFigFont{8}{9.6}{\rmdefault}{\mddefault}{\itdefault}{\color[rgb]{0,.56,0}$k$}%
}}}}
\put(5536,-496){\makebox(0,0)[lb]{\smash{{\SetFigFont{8}{9.6}{\rmdefault}{\mddefault}{\updefault}{\color[rgb]{.5,.17,0}$n$}%
}}}}
\put(5401,344){\makebox(0,0)[lb]{\smash{{\SetFigFont{8}{9.6}{\rmdefault}{\mddefault}{\updefault}{\color[rgb]{0,.56,.56}$m'$}%
}}}}
\put(4321,344){\makebox(0,0)[lb]{\smash{{\SetFigFont{8}{9.6}{\rmdefault}{\mddefault}{\updefault}{\color[rgb]{0,0,.69}$l'$}%
}}}}
\put(  1,-496){\makebox(0,0)[lb]{\smash{{\SetFigFont{8}{9.6}{\rmdefault}{\mddefault}{\itdefault}{\color[rgb]{0,.56,0}$k$}%
}}}}
\put(451,-496){\makebox(0,0)[lb]{\smash{{\SetFigFont{8}{9.6}{\rmdefault}{\mddefault}{\updefault}{\color[rgb]{0,0,.69}$l$}%
}}}}
\put(2071,-496){\makebox(0,0)[lb]{\smash{{\SetFigFont{8}{9.6}{\rmdefault}{\mddefault}{\itdefault}{\color[rgb]{0,.56,0}$k$}%
}}}}
\put(2521,-496){\makebox(0,0)[lb]{\smash{{\SetFigFont{8}{9.6}{\rmdefault}{\mddefault}{\updefault}{\color[rgb]{0,0,.69}$l$}%
}}}}
\put(3061,-496){\makebox(0,0)[lb]{\smash{{\SetFigFont{8}{9.6}{\rmdefault}{\mddefault}{\updefault}{\color[rgb]{0,.56,.56}$m$}%
}}}}
\put(3421,-496){\makebox(0,0)[lb]{\smash{{\SetFigFont{8}{9.6}{\rmdefault}{\mddefault}{\updefault}{\color[rgb]{.5,.17,0}$n$}%
}}}}
\put(991,-496){\makebox(0,0)[lb]{\smash{{\SetFigFont{8}{9.6}{\rmdefault}{\mddefault}{\updefault}{\color[rgb]{0,.56,.56}$m$}%
}}}}
\put(2206,344){\makebox(0,0)[lb]{\smash{{\SetFigFont{8}{9.6}{\rmdefault}{\mddefault}{\updefault}{\color[rgb]{0,0,.69}$l'$}%
}}}}
\put(3286,344){\makebox(0,0)[lb]{\smash{{\SetFigFont{8}{9.6}{\rmdefault}{\mddefault}{\updefault}{\color[rgb]{0,.56,.56}$m'$}%
}}}}
\put(1351,-496){\makebox(0,0)[lb]{\smash{{\SetFigFont{8}{9.6}{\rmdefault}{\mddefault}{\updefault}{\color[rgb]{.5,.17,0}$n$}%
}}}}
\end{picture}%
}}$}
If all the lines are parallel, then by a translation of $k\cup l$ such that
the image of $l$ would coincide with $m$ and applying pencil and involution
relations as above, we can reduce the number of reflections.
\end{proof}
\begin{proof}[\bf Proof of Theorem \ref{thRelCompl}]
Take any relation $F_1\circ F_2\circ \dots\circ F_n=\id$ among reflections.
By Lemma, 
we may reduce its length $n$ to a number which is less than four
by applying pencil and involution relations. The length cannot
become three, because a composition of an odd number of reflections 
reverses the
orientation, and hence cannot be equal to the identity. 
The only composition of two reflections which is the identity is an
involution relation.
\end{proof}

Theorem \ref{thRelCompl} gives a presentation by generators and relations
for the group of plane isometries. The generators in this presentation 
are reflections in all
the lines. This is an uncountable set.  Since the group is uncountable, 
the set of generators must be uncountable, too. Though, this set has 
quite a simple structure. The set of relations has  cardinality of the
continuum, too. It consists of all pencil and involution relations.
Geometrically the relations are easy to deal with. It is so powerful 
that allows to
present any isometry as a product of at most three generators.

Enlarging the system of generators by inclusion all the flips (i.e., by
addition of
symmetries about points) makes it even more powerful: any plane isometry 
admits a presentation as a product of two flips. The set of relations 
has to be enlarged by expressions of symmetries about points as
compositions of reflections: the symmetry about a point $A$ (being a
rotation by $\pi$ about $A$) is the composition of reflections in any two
perpendicular lines intersecting at $A$.

The following theorem ensures an extra comfort in graphical operations in
this presentation of the plane isometry group.


\begin{thm}\label{th1}
Any pair of plane isometires admits a head to tail composition. In other
words, any isometries $S$ and $T$ of Euclidean plane are presented by 
biflippers $\overrightarrow{AB}$ and $\overrightarrow{BC}$, respectively. 
\end{thm}
\begin{proof} 
For any plane isometry, besides a rotation, the tail flipper in a biflipper 
can be chosen to be either a point or a line, as we like, or the head flipper
can be chosen to be either a point or a line. 
(We cannot control the dimensions of both of them simulateously. ) 

Assume first that none of the two
isometries $S$ and $T$ is a rotation. Then there exist biflippers 
$\overrightarrow{KL}$ and $\overrightarrow{MN}$ representing $T$ and $S$,
respectively, such that both $L$ and $M$ are points, or both of them are
lines. By Theorem \ref{th02}, we 
can move each of the biflippers either freely or along a line. If at least 
one of them can be moved freely (the corresponding isometry is a translation)
or the lines meet, one can choose $L$ and $M$ to be points,
superimpose them, and we are done. If the
lines along which the points $L$ and $M$ can move are parallel, 
then we choose $L$ and $M$ to be lines, these lines are parallel, and 
now, by moving the biflippers, we can make the lines superimposed.   

The case of two rotations was considered above. 

If one of the isometries is a rotation, then, by rotating about the center of
the rotation, we can make the lines $L$ and $M$ parallel, and by a move of
the other biflipper superimpose $L$ and $M$.
\end{proof}

\noindent{\bf Exercise. } Find head to tail composition rules for 
the pairs of plane isometries that have not been discussed above.

\section{On the 2-sphere and projective plane}\label{s4}

\subsection{On the 2-sphere.}\label{s4.0}
The group of isometries  of the 2-sphere $S^2=\{x\in\R^3: |x|=1\}$
coincides
with the orthogonal group $O(3)$. The isomorphism maps a linear orthogonal 
transformation of $\R^3$ to its restriction.

As was mentioned in Section \ref{s1}, there are two kinds of flippers on
$S^2$: great circles and pairs of antipodal points. Proper biflippers look as
follows:\\
 \centerline{$\fig{2D-sph-frAr}$}
These biflippers encode all the non-identity isometries of $S^2$. 

\subsection{Rotations.}\label{s4.1}
A biflipper, which is formed by two great circles, defines the rotation 
about the
intersection points of the circles by the angle twice the angle between the
circles. This becomes clear, if we recall that the whole picture is cut on
$S^2$ by the picture in $\R^3$ and the corresponding biflipper in $\R^3$ 
is a pair of planes.\\ 

\noindent\parbox{.45\textwidth}{
\hspace{10pt}
A transition from a biflipper, which is made of great circles $A$ and $B$, 
to a biflipper made of
two pairs of antipodal points is described by Theorem \ref{th02}. For $C$
one has to take the great circle perpendicular to $A$ and $B$.}
\parbox{.55\textwidth}{\ $\vcenter{\hbox{\begin{picture}(0,0)%
\includegraphics{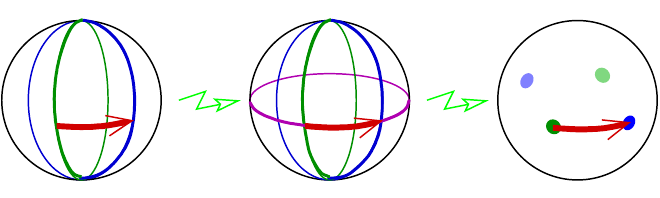}%
\end{picture}%
\setlength{\unitlength}{3729sp}%
\begingroup\makeatletter\ifx\SetFigFont\undefined%
\gdef\SetFigFont#1#2#3#4#5{%
  \reset@font\fontsize{#1}{#2pt}%
  \fontfamily{#3}\fontseries{#4}\fontshape{#5}%
  \selectfont}%
\fi\endgroup%
\begin{picture}(3346,992)(83,-512)
\put(631,-61){\makebox(0,0)[lb]{\smash{{\SetFigFont{9}{10.8}{\rmdefault}{\mddefault}{\updefault}{\color[rgb]{0,0,1}$B$}%
}}}}
\put(1891,-61){\makebox(0,0)[lb]{\smash{{\SetFigFont{9}{10.8}{\rmdefault}{\mddefault}{\updefault}{\color[rgb]{0,0,1}$B$}%
}}}}
\put(1688,140){\makebox(0,0)[lb]{\smash{{\SetFigFont{9}{10.8}{\rmdefault}{\mddefault}{\updefault}{\color[rgb]{.69,0,.69}$C$}%
}}}}
\put(1496,-61){\makebox(0,0)[lb]{\smash{{\SetFigFont{9}{10.8}{\rmdefault}{\mddefault}{\updefault}{\color[rgb]{0,.56,0}$A$}%
}}}}
\put(239,-63){\makebox(0,0)[lb]{\smash{{\SetFigFont{9}{10.8}{\rmdefault}{\mddefault}{\updefault}{\color[rgb]{0,.56,0}$A$}%
}}}}
\end{picture}%
}}$}\vspace{5pt}

Biflippers, that encode rotations and are formed of flippers of the same 
dimension, are equivalent iff they can be 
obtained from each other by a rotation of the sphere about the same axis.
This follows from Theorem \ref{th02}.

Head to tail rule is very simple:\\
\centerline{$\fig{h2tProcSphRot2}$}
It has been recently introduced in the author's preprint \cite{Viro}.

\subsection{Rotary reflections.}\label{s4.2} 
An equivalence class of biflippers for a rotary reflection is determined by an
arc-arrow on a great circle, as for a rotation. The only difference is
that 
in the case of rotation, at both end points of the arc-arrow we have
flippers 
of the same kind: either the end-points (or, rather, two pairs of 
antipodal points), 
or two great circles perpendicular to the arc-arrow, while for a
rotary reflection the flippers are of different kinds: one is a point, the
other is a great circle. It does not matter, which of them is located at
arrowhead and which at the arrow tail. Compare to the relation between
biflippers for glide reflection and translations on the plane.

Rotary reflections and rotations of $S^2$ are related even closer than
glide reflections and translations of $\R^2$. The equivalences for biflippers
in the former case coincide and can be described as obtaining 
arc-arrows from each other by a rotation along the great circle containing
the arc-arrows. In the latter case (on $\R^2$), for translations an
equivalence  class of biflippers is a free vector, while for glide symmetries 
the equivalence class of biflippers is a sliding vector.

Head to tail methods involving rotary reflections are similar to head to tail
methods involving glide reflections. For example, composition of two
rotary reflections is a rotation calculated as follows:\\
\centerline{$\fig{h2tProcSphRotorefl}$}

\subsection{Degenerations.}\label{s4.3}
Rotations and rotary reflections together fill an open everywhere dense
subset of $O(3)$. Reflections in great circles appear as degenerations of
rotary reflections, as the rotation angle vanishes. These degenerations can be
lifted to degenerations of biflipper under which the flippers moves towards
each other and the pair of antipodal points comes to the great circle. 
 
A rotation by $\pi$ is a reflection in two antipodal points (the restriction
of the reflection about a line). The corresponding biflippers are either
a pair of two orthogonal great circles, or two pairs of antipodal points
dividing the great circle passing through them into four congruent arcs.  

 The antipodal symmetry $x\mapsto -x$  
is a rotary reflection, in which the angle of the rotation
part is $\pi$. The corresponding biflippers consist of a great circle and a
pair of antipodal points polar to the circle. Polarity means here that the
great circle is the intersection of $S^2$ with a plane orthogonal to the
line which intersects $S^2$ at the pair of points. 

In the first two degenerations, the equivalence of biflippers and head to
tail methods appear as limits and do not require additional considerations.
In the last degeneration there is only one isometry and all the biflippers
formed by a great circle and a pair of antipodal points polar to each other 
are equivalent. 

\subsection{On the projective plane.}\label{s4.4}

Projective plane considered as the quotient space of $S^2$ by the antipodal
symmetry inherits a metric from $S^2$. Isometries of the projective plane 
are covered by isometries of $S^2$. In particular, flips are covered
by flips. However, flips of the two kinds are covered by
flips of the same type, with the flipper consisting of a projective line
and its polar - the point the most distant from the projective line.

The whole flipper is determined by this point - the rest of the flipper can
be recovered as the polar of the point.
Any isometry is determined by a biflipper, which, in turn, is determined by 
an arrow connecting the 0-dimensional parts of the mirrors. 

Sliding of an arrow along its line does not change the equivalence class of
the corresponding biflipper. Therefore for composing of isometries, one can
use the obvious  head to tail rule.  

\section{ Self-contained digression on quaternions}\label{s4.5}

\subsection{Quaternions.}
Any story about rotations of the 3-space would be incomplete without
quaternions. Unit quaternions provide a convenient para\-met\-ri\-zation for the
group of rotations. Biflippers give another presentation of the same group.
The quaternion parametrization is more convenient for calculations. The
biflippers seems to be more visual and intuitive. We will relate the two
pictures below.  In this section, all the information about quaternions, that is
needed for this, is presented in all the details with brief, but complete proofs.

Quaternions form a 4-dimensional associative algebra over the field of real
numbers. The algebra of quaternions is denoted by  $\mathbb H$. 
As a vector space over $\R$, it has the standard basis $1,i,j,k$. 
The generators are subject to relations $i^2=j^2=k^2=ijk=-1$. 
A quaternion expanded in the standard basis  
is $a+bi+cj+dk$, where $a,b,c,d\in\R$. 
The quaternion addition is component-wise. The quaternion 
multiplication is associative
and the products of the generators are calculated according to formulas
$ij=k$, $ji=-k$, $jk=i$, $kj=-i$, $ki=j$ and $ik=-j$ (these formulas follow 
from the relations  $i^2=j^2=k^2=ijk=-1$).     

\subsection{Scalars and vectors.}
The field $\R$ is contained in $\Qu$ as $\{a+0i+0j+0k\mid a\in\R\}$. \ A
quaternion of the form $a + 0i + 0j + 0k$, \ is called {\sfit real\/}. \ 
A quaternion of the form $0 + bi + cj + dk$, where $b, c, d\in\R$ 
is called {\sfit pure imaginary\/}. If $q=a + bi + cj + dk$ is any quaternion, then $a$ 
is called its {\sfit scalar part\/} and denoted by $q_s$ and 
$bi + cj + dk$ is called its {\sfit vector part\/} and denoted by $q_v$. 
The set of purely imaginary quaternions  $bi + cj + dk$ 
is identified with the real 3-space $\R^3$. 

\subsection{Multiplication of quaternions.}
The set of real quaternions is the center of $\Qu$. Multiplication of
quaternions is composed of all the standard multiplications of factors which
are real numbers and vectors: multiplications of real numbers,
multiplication of a vector by a real number and dot and cross products of
vectors. It is not accident: the very notion of vector and all the
operations with vectors were introduced by Hamilton after invention of
quaternions.

\noindent
{\bf Quaternion product of vectors.} {\sfit Let
$p=ui+vj+wk$ and $q=xi+yj+zk$ be vector quaternions. Then
$pq=-p\cdot q+p\times q$. \/}
Indeed, \vspace{-10pt}
\begin{multline*}pq=
(ui+vj+wk)(xi+yj+zk)\\=-(ux+vy+wz)+(vz-wy)i+(wx-uz)j+(uy-vx)k\\
=-p\cdot q+p\times q \qed 
\end{multline*}

\noindent
{\bf Product of quaternions via other products.} {\sfit For any $p,q\in\Qu$}\\ 
\vspace{-15pt}
\begin{multline*}
pq=(p_s+p_v)(q_s+q_v)=p_sq_s+p_sq_v+p_vq_s+p_vq_v\\
=p_sq_s+ p_sq_v+q_sp_v-p_v\cdot q_v+p_v\times q_v\\
=p_sq_s-p_v\cdot q_v + p_sq_v+q_sp_v+p_v\times q_v \qed
\end{multline*}  

\subsection{Conjugation.}
The map  $\Qu\to\Qu:q\mapsto q^*=q_s-q_v$ is called {\sfit conjugation.\/} The
conjugation is an antiautomorphism of $\Qu$ in the sense that it is an
automorphism of $\Qu$ as a real vector space and $(pq)^*=q^*p^*$. The latter
is verified as follows:\vspace{-10pt}
\begin{multline*}
(pq)^*=\left(p_sq_s-p_v\cdot q_v + p_sq_v+q_sp_v+p_v\times q_v\right)^*\\
=p_sq_s-p_v\cdot q_v - p_sq_v-q_sp_v-p_v\times q_v\\
=p_sq_s-(-p_v)\cdot(- q_v)+p_s(-q_v)+q_s(-p_v)+(-q_v)\times(-p_v)\\
=(q_s-q_v)(p_s-p_v)=q^*p^*.
\end{multline*}  

\subsection{Norm.} 

 The product $q^*q$ is a non-negative real number
for any non-zero quaternion $q$.
Indeed, $(q^*q)^*=q^*(q^*)^*=q^*q$.

If \ $q=a+bi+cj+dk$, then \ $q^*q=a^2+b^2+c^2+d^2$. \ Indeed, \ 
$q^*q=q_s^2-(-q_v)\cdot q_v+q_sq_v+q_s(-q_v)+(-q_v)\times q_v=q_s^2+q_v\cdot
q_v=a^2+b^2+c^2+d^2$. 

The number $\sqrt{q^*q}$ is called the {\sfit norm\/}
of $q$ and denoted by $|q|$. It is the Euclidean distance from $q$ to the
origin in $\R^4$. \smallskip

\noindent
{\sfit The norm is a multiplicative homomorphism $\Qu\to\R$.\/} \
Indeed,
$|pq|=\sqrt{pq(pq)^*}=\sqrt{pqq^*p^*}=\sqrt{p(qq^*)p^*}=\sqrt{pp^*}\sqrt{qq^*}=|p||q|$.

\section{The group of unit quaternions.}\label{s4.6}

\subsection{Unit quaternions.}
The sphere $S^3=\{q\in\Qu\mid|q|=1\}$, being the kernel of the
multiplicative homomorphism $\Qu\to\R:q\mapsto|q|$, is a multiplicative subgroup of $\Qu$. 
The inverse to a
quaternion $q\in S^3$ coincides with $q^*$. Indeed, $|q|=\sqrt{qq^*}=1$,
hence $qq^*=1$.

Unit vector quaternions form the unit 2-sphere $S^2$ in $\R^3$. It 
is contained in $S^3$ as an equator. The unit vectors are very special
quaternions.

\begin{thm}\label{vectGen}
Each unit quaternion can be presented as a product of two unit vectors.
Moreover, if $q$ is a unit quaternion and $v$ is a unit vector
perpendicular to $q_v$, then there exist unit vectors $w_+$ and $w_-$ such
that $q=vw_+=w_-v$.
\end{thm}
In particular, unit vectors generate the group of unit quaternions. 
\begin{proof} Let $q\in S^3$ be a unit quaternion. Then $q=q_s+q_v$ with
$1=|q|^2=q_s^2+|q_v|^2$. Choose $\Ga\in[0,\pi]$ such that $q_s=\cos\Ga$
and $|q_v|=\sin\Ga$. Then $q=\cos\Ga+u\sin\Ga$ for some unit vector $u$. 

Take any unit vector $v$ perpendicular to $u$. Then $w_+=-v\cos\Ga+(u\times
v)\sin\Ga$ and $w_-=-v\cos\Ga-(u\times v)\sin\Ga$ are also unit vectors 
perpendiculars 
to $u$, with the required properties: $vw_+=q$ and $w_-v=q$. Indeed, 
$vw_+=v(-v\cos\Ga+(u\times v)\sin\Ga)=-v\cdot(-v\cos\Ga)+v\times(u\times
v)\sin\Ga=\cos\Ga+u\sin\Ga=q$ and $w_-v=(-v\cos\Ga-(u\times v)\sin\Ga)v=
-(-v\cos\Ga)\cdot v-(u\times v)\times v\sin\Ga=\cos\Ga+u\sin\Ga=q$  
\end{proof}

\noindent
{\bf Remark.} {\sfit Any unit vector quaternion $u$ has order four, its
multiplicative inverse coincides with the additive inverse: $u^{-1}=-u$.}\\
Indeed, let $u$ be unit vector. Then $u^2=-u\cdot u+u\times u=-1$, 
hence $u^3=-u$ and $u^4=(u^2)^2=(-1)^2=1$.\qed

\subsection{Unit quaternions as fractions of unit vectors.} By Theorem
\ref{vectGen} any unit quaternion $q$ admits presentation as product of two 
unit vector quaternions: $q=vw$. If the factors $v$ and $w$ were of order
two, we could use this presentation for head to tail rule. However the unit 
vector quaternions have order four, and we need to modify the presentation 
slightly, by replacing the product with a quotient. 

Namely, let us present a unit quaternion as a sort of {\sfit quotient\/} of two unit
vectors: $q=v^{-1}w=-vw$. By the way, this goes back to W.R.Hamilton, 
the inventor of quaternions. In his book \cite{Hamilton}, Hamilton introduced 
quaternions as quotients of vectors. 
 
\subsection{Hamilton's model for the group of unit quaternions.}
The underlying space is the 2-sphere $S^2$. Unit quaternions are
interpreted as classes of equivalent vector-arcs on $S^2$. A vector-arc 
on $S^2$ is an ordered pair of points of $S^2$ connected by an arc of great
circle and equipped with an arrowhead pointing to the second of the points.
Two vector-arcs 
are equivalent iff they either lie on the same great circle and can be 
obtained from each other by an (orientation preserving) rotation of the
great circle, or connect a pair of antipodal points. 

Multiplication is defined by the head to tail rule: given two equivalence 
classes of vector-arcs, find the intersection point of two great circles 
containing vector-arcs of the classes, find the representatives in the head
to tail position and draw the vector-arc connecting the tail of the first
representative with the head of the second. See Hamilton \cite{Hamilton}, 
Book II, Chapter I, Section 9. \\
\centerline{$\fig{h2tProcSphRot}$}
The unit with respect to this multiplication is the class of all vector-arcs 
of the zero length. 

The vector-arc connecting points $v$ and $w$ on $S^2$ represents the unit
quaternion $-vw$. Equivalent vector-arcs represent the same unit quaternion.
The product of the quaternions $-vw$, $-wx$ representing vector-arcs, which
are in head to tail position, equals $(-vw)(-wx)=vw^2x=v(-1)x=-vx$, the
quaternion corresponding to the vector-arc obtained by the head to tail
rule. 
\smallskip

\noindent
{\bf Remark 1.} Our model of the group $S^3$ is not based on flip-flop 
decomposition
for element of $S^3$ acting in itself. Indeed, $S^3$ is not
generated by involutions. There is only one involution in the group $S^3$, the
quaternion $-1$. It belongs to the center.  
A single point, i.e., an end point of an arrow
in the model of $S^3$ corresponds to an element of order 4. In the Hamilton 
model, such
elements are represented by vector-arcs occupying a quarter of a great circle.
\smallskip

\noindent
{\bf Remark 2.} The groups $S^3$ of unit quaternions is isomorphic to
$SU(2)$. The standard isomorphism maps $a+bi+cj+dk\in S^3$ to
$\begin{pmatrix}a+bi&c+di\\-c+di&a-bi\end{pmatrix}$. Thus, the Hamilton model
describes also $SU(2)$.\smallskip

\noindent
{\bf Remark 3.} A similar model on the circle $S^1$ describes $Pin^-(2)$.

\section{Two stories about one two-fold covering}\label{s4.7}

\subsection{The covering.}
The group $SO(n)$ with $n\ge3$ is known to have a unique two-fold covering. 
The covering space is known as $Spin(n)$, it is also a Lie group. For
$n=3$, the latter group is known to coincide with the group $S^3$ of unit 
quaternions. There are many descriptions of the covering $S^3\to SO(3)$. 
The Hamilton model of $S^3$ coupled with biflippers of rotations on $S^2$ 
provides probably the simplest description.

\subsection{A biflipper view on the covering.}
The description of rotations of $S^2$ via biflipper given above in Section 
\ref{s4} can be converted obviously into a model for $SO(3)$ similar to the
Hamilton's vector-arcs model for unit quaternions, see section \ref{s4.6}.

In this model, elements of $SO(3)$  are interpreted as classes of equivalent 
0-dimensional biflippers on $S^2$. Recall that a 0-dimensional biflipper on
$S^2$ is two pairs of antipodal points connected with an arc-arrow, and
two such biflippers are equivalent iff  they lie on the same great circle 
and can be obtained from each other by a rotation of the great circle.

There is an obvious two-fold covering, in which the base is this model of 
$SO(3)$ and the total space is the Hamilton model for the group of unit quaternions.
Namely, each end point of an vector-arc is completed by its antipode.
The preimage of a biflipper consists of four vector-arcs, but they split into
two pairs of equivalent vector-arcs.\\ 
\centerline{$\vcenter{\hbox{\begin{picture}(0,0)%
\includegraphics{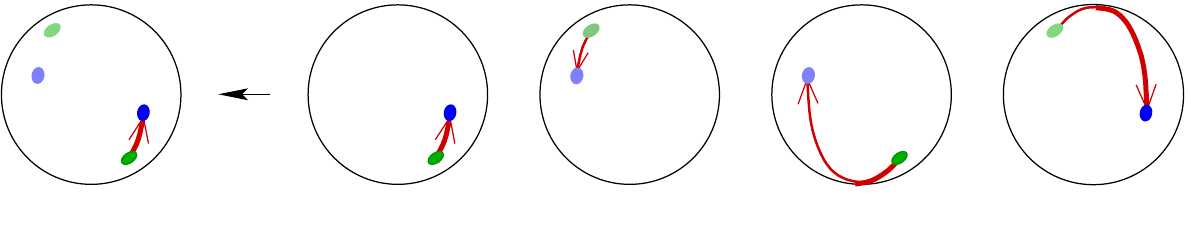}%
\end{picture}%
\setlength{\unitlength}{3149sp}%
\begingroup\makeatletter\ifx\SetFigFont\undefined%
\gdef\SetFigFont#1#2#3#4#5{%
  \reset@font\fontsize{#1}{#2pt}%
  \fontfamily{#3}\fontseries{#4}\fontshape{#5}%
  \selectfont}%
\fi\endgroup%
\begin{picture}(7128,1428)(-905,-575)
\put(-674,-466){\makebox(0,0)[lb]{\smash{{\SetFigFont{9}{10.8}{\rmdefault}{\mddefault}{\itdefault}{\color[rgb]{0,0,0}in $SO(3)$}%
}}}}
\put(2881,-511){\makebox(0,0)[lb]{\smash{{\SetFigFont{9}{10.8}{\rmdefault}{\mddefault}{\itdefault}{\color[rgb]{0,0,0}in $S^3$}%
}}}}
\put(2071,254){\makebox(0,0)[lb]{\smash{{\SetFigFont{9}{10.8}{\rmdefault}{\mddefault}{\itdefault}{\color[rgb]{0,0,0}$\sim$}%
}}}}
\put(4861,254){\makebox(0,0)[lb]{\smash{{\SetFigFont{9}{10.8}{\rmdefault}{\mddefault}{\itdefault}{\color[rgb]{0,0,0}$\sim$}%
}}}}
\end{picture}%
}}$}
This is a group homomorphism, since in both models
the multiplication is defined by the same head to tail rule.

\subsection{The action of unit quaternions in the 3-space.}
A traditional description of the same covering looks as follows.

The group $S^3$ of unit quaternions acts in $\Qu$ by formula $
\rho_q(p)=qpq^{-1}=qpq^*$. 

{\sfit This action commutes with the conjgation $p\mapsto p^*$.\/} Indeed,
$\rho_q(p^*)=qp^*q^*=((q^*)^*(p^*)^*q^*)^*=(qpq^*)^*=(\rho_q(p))^*$.

Therefore the action of $S^3$ in $\Qu$ preserves all the structures defined by
the congugation: the norm and the decomposition into scalar and vector
parts. Indeed,
$$
\rho_q(p_v)=\rho_q\left(\frac{p-p^*}2\right)=\frac{\rho_q(p)-\rho_q(p^*)}2=\frac{\rho_q(p)-\rho_q(p)^*}2=(\rho_q(p))_v,
$$
$$
\rho_q(p_s)=\rho_q\left(\frac{p+p^*}2\right)=\frac{\rho_q(p)+\rho_q(p^*)}2=\frac{\rho_q(p)+\rho_q(p)^*}2=(\rho_q(p))_s,
$$
\begin{multline*}|\rho_q(p)|=\sqrt{\rho_q(p)(\rho_q(p))^*}=\sqrt{\rho_q(p)\rho_q(p^*)}\\
=\sqrt{(qpq^*)(qp^*q^*)}=\sqrt{qpp^*q^*}=|p|.\end{multline*}
 In particular, the space $\R^3$ of vector quaternions is invariant,
and $S^3$ acts on $\R^3$ by isometries. The action is a
homomorphism $S^3\to O(3)$. Since $S^3$ is connected, the image of this
homomorphism lies in $SO(3)$.

\begin{thm}\label{tlem1}
A unit vector quaternion $v$ acts in $\R^3$
as the symmetry about the line generated by $v$.
\end{thm}

\begin{proof} The statement that we are going to prove admits the following
reformulation: for the linear operator $\R^3\to\R^3:u\mapsto vuv^*$, the 
vector $v$ is mapped to itself and each unit vector $u$ orthogonal to $v$ 
is an eigenvector with eigenvalue $-1$. 

Let us verify the first statement. Since $v$ is a unit
vector, $vv^*=|v|^2=1$. Therefore $vvv^*=v$.

Now let us verify the second statement. Since $u$ is a unit vector
orthogonal to $v$, $vu=v\times u-v\cdot u=v\times u$. Therefore, 
 $vuv^*=-vuv=-(v\times u-v\cdot u)v=-(v\times u)v$. Vector $v\times u$ is
orthogonal to $v$. Therefore $-(v\times u)v=-(v\times u)\times v+(v\times
u)\cdot v=-(v\times u)\times v=-u$. The latter equality holds true, because
$(a\times b)\times a=b$ for any orthogonal unit vectors $a,b$ (e.g.,
$(i\times j)\times i=k\times i=j$).  
\end{proof}

\begin{thm}[Euler-Rodrigues-Hamilton]\label{thEulerRodrigesHamilton}
For any unit quaternion $q\in S^3$ the map
$
\R^3\to\R^3: p\mapsto qpq^*
$
is the rotation of $\R^3$ about the axis generated by a unit vector $u$ by the angle
$\theta$, where $\theta$ and $u\in\R^3$ are 
such that $q=\cos\frac\theta2+u\sin\frac\theta2$.  
\end{thm}

\begin{proof}
By Theorem \ref{vectGen} any unit quaternion $q$ can be presented as a product
of unit vectors $v$ and $w$. In this proof it will be more convenient to
use a modification of this presentation, the fraction presentation
$q=v^{-1}w=-vw$ discussed above.
 
 By Theorem \ref{tlem1} a unit vector acts as a
symmetry about the line generated by this vector. Thus,  $\rho_q$
is the composition of the symmetries $\rho_{-v}$ and $\rho_{w}$. 
We know that the
composition of symmetries about lines is a rotation by the angle equal the
half of the angle between the lines. On the other hand, 
$q_s=(v(-w))_s=-v\cdot(-w)=v\cdot w=\cos\Ga$, where $\Ga$ is the angle
between the vectors $v$ and $w$. Thus $q_s=\cos\frac\theta2$, where $\theta$ is the
rotation angle. 

The vector part $q_v$ of the product of two unit vectors $v$ and $-w$ is
collinear to $v\times(-w)$. The cross product of
vectors is perpendicular to the vectors. On the other hand, we know that
composition of symmetries about lines is a rotation about the axis
perpendicular to the lines. Thus the vector $q_v$ is collinear to
the axis of the rotation $\rho_q$. 
The length $|q_v|$ is $|\sin\frac\theta2|$, because $|q|=1$ and
$q_s=\cos\frac\theta2$. Therefore $q_v=u\sin\frac\theta2$ for some unit 
vector $u$ collinear to the axis of rotation.  
\end{proof}

The quaternion $q$ can be written down as $a+bi+cj+dk$. It is defined by
the rotation up to multiplication by $-1$. 
Its  components $a,b,c,d$ are called the {\sfit Euler parameters\/} for this
rotation. They are calculated as follows: 
$a=\cos\frac\theta2$, $b=u_x\sin\frac\theta2$, $c=u_y\sin\frac\theta2$ 
and $d=u_z\sin\frac\theta2$, where $u_x$, $u_y$ and $u_z$ are
coordinates of the unit vector $u$ directed along the rotation axis.

Given a 0-dimensional biflipper $\overrightarrow{AB}$ defining a rotation, 
the Euler parameters of the rotations can be recovered just by choosing 
one point from $A$ and $B$, say $v\in A$ and $w\in B$ (recall that $A$ and
$B$ are pairs of antipodal points on $S^2$), and quaternion multiplication
of the representative: $q=vw$. The components of $q$ are the Euler
parameters of the rotation.

\section{On the hyperbolic plane}\label{s5}

\subsection{Well-known facts.}\label{s5.0}
We will use the Poincar\'e model, in which the hyperbolic plane is
represented by an open unit disk in $\R^2$.
In this model, a line is either a diameter of the disk or an arc cut on 
the disk by a circle orthogonal to the boundary circle of the disk. 
The boundary circle of the disk is the {\sfit absolute\/.} (The absolute is
not contained in the hyperbolic plane, it is rather something like a horizon.)

\begin{wrapfigure}[6]{r}{2.5in}
$\fig{linesOnHypPl}$
\end{wrapfigure}
On the hyperbolic plane, two lines may intersect in a single point, 
or be disjoint. 
In the latter case, their closures
on the absolute may be disjoint, and then the lines are said to be 
{\sfit ultra-parallel\/}, or have a common point, and then the lines are said
to be {\sfit parallel\/}. 

Two ultra-parallel lines have a unique common perpendicular. Conversely, 
two lines perpendicular to the same line are ultra-parallel. The set of 
lines perpendicular to the same line is called a {\sfit hyperbolic pencil\/} 
of lines. The line which is their common perpendicular is called the axis
of the pencil.

The set of all lines passing through the same point, is called an {\sfit
elliptic pencil\/} of lines. Their common points is called the center of
the pencil.

The set of all lines whose closures share the same point on the absolute
is called a {\sfit parabolic pencil\/} of lines. Their common point on the
absolute is called the center of the pencil.

Observe that lines of a pencil of any type fill the whole hyperbolic plane,
and only one line of the pencil passes  through any point except the center
of an elliptic pencil.   

There are two kinds of flippers:  lines and
 points. In the Poincar\'e model, the flip in a line, 
which is a diameter of the disk, is the restriction of the usual Euclidean 
symmetry with respect to the line of the diameter; the flip in a line, 
which is cut by a circle $C$, is a restriction  of the plane inversion with 
respect to $C$. A flip in a point is a composition of
flips in any two lines passing through this point and orthogonal to
each other. Any line and any point is a flipper.

\subsection{Biflippers.}\label{s5.1}
Proper biflippers can be formed of two points, a point and a line, which
may be disjoint or the point may belong to the line, and a pair of lines. 
In the Poincar\'e model of the hyperbolic plane proper biflippers look as
follows.\\ 
\centerline{$\fig{2D-hyp-frAr}$}

A biflipper, which encodes a rotation or parallel motion, consists of
lines. 
Except rotations and parallel motion, each isometry admit biflippers
that include a mirror of dimensions 0, and one can use the
construction of Theorem \ref{th01} in order to change the dimensions of
both flippers simultaneously and get an equivalent biflipper.  

Biflippers can be varied in their equivalence classes also in a continuous
manner according to Theorem \ref{th02}. 

\begin{thm}\label{th03}
For any isometry $S$ of the
hyperbolic plane, there is a pencil $\mathcal P(S)$ of lines, which is invariant under any 
isometry from the centralizer $C(S)$ of $S$. Any one-dimensional flipper 
from any biflipper, which encodes $S$, belongs to  $\mathcal P(S)$. 
Moreover, $C(S)$ acts transitively in the set
of lines belonging to  $\mathcal P(S)$: for any $L,M\in\mathcal P(S)$  
there exists an $T\in C(S)$  such that
$T(L)=M$. In particular, for any $L\in\mathcal P(S)$ there are biflippers
$\overrightarrow{AL}$ and $\overrightarrow{LB}$, which encode $S$. 
\end{thm}

\begin{proof} If $S$ is a translation, then it has a unique invariant line. It
is called the axis of $S$. In this case, $\mathcal P(S)$ is the hyperbolic
pencil of all lines perpendicular to the invariant line. 

If $S$ is parallel motion, then the induced map of the absolute has a unique
fixed point. Then $\mathcal P(S)$ is the parabolic pencil of all lines 
whose closure contains this point.

If $S$ is a rotation, then it has a fixed point, and $\mathcal P(S)$ is the
elliptic pencil of all lines passing through this point.

If $S$ is a glide reflection, then (as in the case of translation) 
it has a unique invariant line called the axis of $S$ and  $\mathcal P(S)$
is the corresponding hyperbolic pencil of lines.
 
If $S$ is a reflection in a line, then  $\mathcal P(S)$ is the hyperbolic
pencil of all lines perpendicular to this line. 
\end{proof}

{\sffamily
\subsection*{\sffamily Added in August 2015.}\label{s5a6}  
Compare to the notions of {\sfit
hyperbolic pencil} and {\sfit half-turn banks} introduced by Silverio in
\cite{Silverio}.}

\subsection{Head to tail methods.}\label{s5.2}
Let $S$ and $T$ be isometries of the hyperbolic plane. In this section we
describe how to find a biflipper for $S\circ T$, given biflippers of $S$ and
$T$.

If $S$ and $T$ are rotations, then we act as in the case of two rotations in
$\R^2$. Connect the fixed points of $S$ and $T$ by a line $L$. It belongs
to both $\mathcal P(S)$ and $\mathcal P(T)$. By Theorem \ref{th03} there
exist biflippers $\overrightarrow{AL}$ and $\overrightarrow{LB}$, which
encode $T$ and $S$, respectively. The biflipper $\overrightarrow{AB}$ encodes
$S\circ T$.

If $S$ or $T$ is a rotation and the other of them is not, then we act as in
the case of rotation and translation. Let, say, $S$ be a rotation. 
Choose a line from the $\mathcal P(T)$, which passes through the center of
$S$ and hence belong to $\mathcal P(S)$. Then do the same as above.

Any other situation can be easily reduced to one of these two. Choose any
biflippers $\overrightarrow{AB}$ and $\overrightarrow{CD}$ representing $T$
and $S$.  If $A$ is not a line, change both $A$ and $B$ to make $A$ a
line. 

Consider the biflipper  $\overrightarrow{BC}$ and the corresponding
isometry $U=F_C\circ F_B$. Lines of the pencil $\mathcal P(U)$ cover the
whole plane. Choose a point $Q\in A$ and a line $L\in\mathcal P(U)$ passing
through $Q$. By Theorem \ref{th03}, there is a biflipper
$\overrightarrow{LM}$ which encodes $U$. So, 
$$
S\circ T=F_D\circ F_C\circ F_B\circ F_A=F_D\circ U\circ F_A=F_D\circ
F_M\circ F_L\circ F_A
 $$
Since $L$ and $A$ intersect, either $L=A$ and then $S\circ T=F_D\circ F_M$
(and we are done), or $F_L\circ F_A$ is a rotation and we are in the
situation considered above. 

However, in most cases this reduction to rotations is not needed
and there is a direct simpler construction. Consider those cases. 

Let $S$ and $T$ be parallel motion. Then the pencils $\mathcal P(S)$
and $\mathcal P(T)$ intersect (a common line is connecting the points on
the absolute which are the centers of these pencils), and this line can be
made the head of one biflipper and the tail of the other one.

Let both $S$ and $T$ have invariant lines $L$ and $M$ (i.e., each of $S$
and $T$ is either a translation, or a glide reflection, or a reflection) and $L$, $M$ be non
parallel lines distinct from each other. Then either $L$ and $M$ intersect, 
or they are ultra parallel. In the latter case the pencils $\mathcal P(S)$
and $\mathcal P(M)$ have a common line and it can be made the head of a
biflipper encoding $T$ and tail of the biflipper encoding $S$. In the former
case, an intersection point of $L$ and $M$ can be made the head of a
biflipper encoding $T$ and tail of the biflipper encoding $S$. 

A hyperbolic pencil of lines orthogonal to line $L$ and parabolic pencil of
lines with central point $Q$ on the absolute which does not belong to (the
closure of) $L$ have a common line. This allows to achieve head to tail
biflippers in the cases when one of the isometries is a parallel motion and the
other one has invariant line not containing the center of the parabolic
pencil of the other one.  

Thus, the only situations, in which there is no biflippers for $S$ and $T$ in
head to tail position, are when either both $S$ and $T$ have invariant lines
and these lines are parallel, or one of the isometries is parallel motion the
other has an invariant line and this line is parallel to the lines of the
pencil of lines of the parallel motion. These are degenerate situations
and for them one has to make a preliminary replacement of $S$ and $T$ which
would not affect $S\circ T$. The preliminary replacement can be done so
that $T$ would be replaced by a rotation, as it was described above.
\newpage
 
\section{Back to Euclidean spaces}\label{s6}\nopagebreak
\subsection{On line.}\label{s6.1}
Any isometry of a line is either the reflection in a point or a translation.
A translation is encoded by a biflipper made of two points. A reflection in
a point can be presented by a biflipper made of the point and the whole line,
no matter in which order. 
\\   
\centerline{\includegraphics{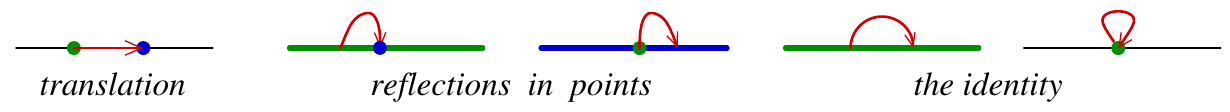}}

\subsection{Multiplication of biflippers.}\label{s6.2}
For any isometries $S:\R^k\to \R^k$ and $T:\R^l\to\R^l$, the direct product 
$S\times T:\R^k\times\R^l\to\R^k\times\R^l:(x,y)\mapsto(S(x),T(y))$ is an
isometry of $\R^{k+l}=\R^k\times\R^l$. 
If $A\subset\R^k$ and $A'\subset\R^l$ are affine subspaces, then $A\times
A'$ is an affine subspace of $\R^k\times\R^l=\R^{k+l}$ and $F_A\times
F_{A'}=F_{A\times A'}$. 

\begin{thm}\label{th2}
For any biflippers $\overrightarrow{AB}$ and $\overrightarrow{CD}$, which 
encode isometries $S:\R^k\to\R^k$ and $T:\R^l\to\R^l$, the isometry
$S\times T:\R^{k+l}\to\R^{K+l}$ is encoded by 
$\overrightarrow{(A\times C)(B\times D)}$.  
\end{thm}
\begin{proof}
Notice that if $F,F':\R^k\to\R^k$ commute and $G,G':\R^l\to\R^l$ commute,
then $F\times G$ commutes with $F'\times G'$. Therefore
\begin{multline*}
F_{B\times D}\circ F_{A\times C}\\
= (F_B\times \id)\circ(\id\times
F_D)\circ(F_A\times\id)\circ(\id\times F_C)\\
=(F_B\times\id)\circ(F_A\times\id)\circ(\id\times F_D)\circ(\id\times
F_C)\\
=((F_B\circ F_A)\times\id)\circ (\id\times(F_D\circ F_C))\\=
(F_B\circ F_A)\times(F_D\circ F_C)=S\times T.
\end{multline*}
\end{proof}

\begin{cor}\label{cor1th2}
Let $A$, $B$ be affine subspaces of $\R^k$, and $C$ be an affine subspace
of $\R^l$. Then the biflipper
 $\overrightarrow{(A\times C)(B\times C)}$ encodes the direct
product of the isometry of $\R^{k}$ encoded by $\overrightarrow{AB}$ by the
identity of $\R^l$ (no matter what $C$ is). \qed
\end{cor}

Observe, that all isometries of the plane, except rotations by angles
$0<\Gf<\pi$, can be obtained as products of isometries of the line. Their
biflippers can be obtained as products as well.

\subsection{In Euclidean $n$-space.}\label{s6.3}
As follows from the well-known classification of isometries of $\R^n$ (see,
e.g., M.~Berger \cite{Berger}), any isometry of $\R^n$ is isometric to a direct
product of isometries in factors of dimension at most two. Therefore, any
isometry of $\R^n$ can be encoded by a biflipper. 

This statement can be obtained also as  a corollary of much more general 
results by Wonenburger \cite{W}. She proved, in particular, that {\sfit 
any isometry of a 
non-degenerate inner product vector space over any field of characteristic
$\ne2$ can be presented as a 
composition of two linear involutory isometries.}
 
From this one can easily deduce that 
{\sfit any isometry of an affine space over a
field of characteristic $\ne2$ with 
a non-degenerate bilinear form can be
presented as a composition of two flips.\/} 

This implies that in a classical space forms (simply connected complete 
spaces of
constant curvature), isometries can be presented as compositions of two 
flips (i.e. by biflippers).

 One can ask if there exist 
simple and useful descriptions for biflippers of the composition of
isometries presented by biflippers. \medskip


\section{In  Euclidean 3-space}\label{s7}



Biflippers in the 3-space  look as follows:\vspace{5pt}

\centerline{\includegraphics{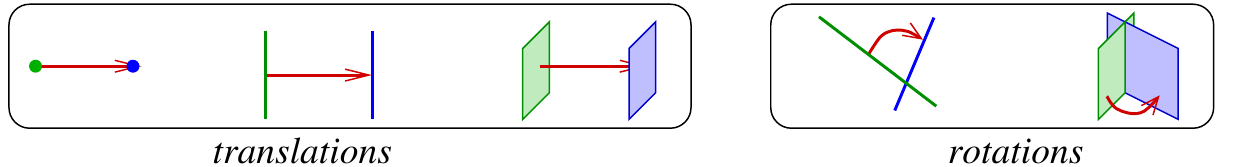}}
\vspace{15pt}

\centerline{\includegraphics{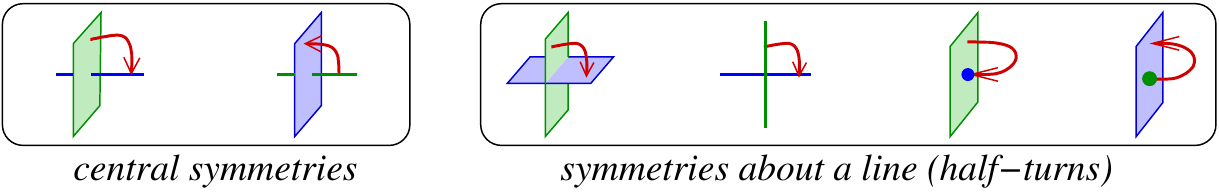}}
\vspace{15pt}

\centerline{\includegraphics{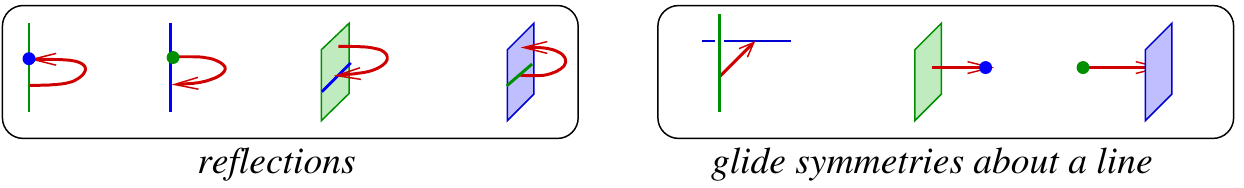}}
\vspace{15pt}

\centerline{\includegraphics{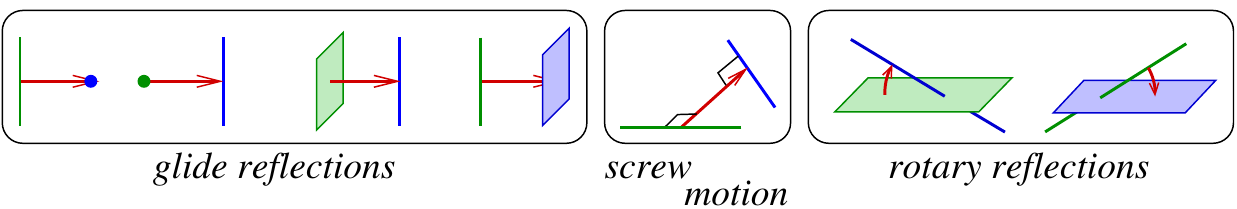}}
\vspace{10pt}

Here really all possible mutual positions of two proper distinct affine 
subspaces of $\R^3$ are presented. 
We take into account if one of the subspaces is
contained in the other or not, if the subspaces are parallel, 
perpendicular or form a generic angle, if the lines are skew.  

\subsection{Decompositions into products.}\label{s7.2}
Each of these biflippers can be obtained as products of biflippers in
plane and line. Therefore one can use Theorem \ref{th2} for recognizing the 
isometries encoded by the biflippers. Moreover, the biflippers
for translations, rotations, reflections, glide reflections and two of
four biflippers for symmetries about a line satisfy the assumptions 
of Corollary of 
Theorem \ref{th2} and encode a product of a plane isometry by the identity
map of the line. The reader can easily verify this.

\noindent{\bf Central symmetries. } 
Biflippers which encode the central symmetries are products of biflippers 
encoding central symmetries of plane and line. For instance:\\
\centerline{$\fig{0-refl-prod}$}\smallskip

\noindent{\bf Symmetries about a line. }
 Two biflippers which encode 1-flips and are formed by a plane and a 
point are products of biflippers 
which encode reflections of plane and line. For instance:\\ 
\centerline{$\fig{1-refl-prod}$}\smallskip

The two biflippers which encode symmetries about a line of the 3-space are 
degenerations of two biflippers for glide symmetries about a line.  
On our picture of all 3D biflippers, 
the biflippers of glide symmetries about a line are right under the 
corresponding biflippers for symmetries about a line.\\
 
\noindent{\bf Glide symmetries about a line. }
These biflippers for glide symmetries about a line are direct
products of biflippers for a glide reflection of plane and a reflection
of line. \\
\centerline{$\fig{gl-1-refl-prod}$}\\
\smallskip

\noindent{\bf Screw motions. }
A screw motion is a direct product of a plane rotation by a
translation of a line. 
By Theorem \ref{th2}, a biflipper for a screw motion
can be obtained as a product of biflippers for rotation and
translation:\\
\centerline{$\fig{skr-prod}$}\smallskip

\noindent{\bf Rotary reflections. } 
A rotary reflection is a direct product of a plane rotation by a reflection
of line in a point. The biflippers for rotary reflections drawn above can
be obtained (by Theorem \ref{th2}) as products of biflippers for
rotation and reflection. For instance:\\
\centerline{$\fig{rotorefl-prod}$}\smallskip

\subsection{Generic isometires.}\label{s7.3} Isometries of the last two types,
namely, screw
motions and rotary reflections, are generic in the following sense.
 
The space of all isometries of $\R^3$ is a Lie group of dimension 6. It
consists of two connected components. The connected component, which
contains the identity, consists of isometries preserving
orientation; the other component consists of isometries reversing orientation. 
Screw motions constitute an open dense subset in the set of orientation 
preserving isometries and rotary reflections  constitute an open dense subset 
 in the set of orientation reversing isometries. 

For any isometry of
$\R^3$, there exists a biflipper, which can be obtained as a
limit of biflippers of screw motions or rotary reflections.  
There may be also biflippers that are not such limits. 

For example, a biflipper of a rotation, which is formed by two intersecting
lines, can be presented as a limit of biflippers formed by skew lines, that is
biflippers of screw motions. A biflipper of a rotation, which 
is formed by two
planes, is not a limit of biflippers of screw motions. 

In the study of biflippers in $\R^3$, we will restrict ourselves 
to the case of generic isometries. We leave to the reader extensions of the 
results to more special isometries. At least, for biflippers that are limits
of biflippers of screw motions, it can be easily done by passing to
the limit.
\smallskip

\subsection{Screw motions.}\label{s7.5} 
As we saw above, an ordered pair of skew lines  \begin{wrapfigure}{l}{1in}
$\fig{skr-prod2}$ 
\end{wrapfigure}
 is a biflipper for a screw motion. Recall that a screw
motion is the composition of a translation with a rotation about a
line parallel to the direction of the translation. For the composition of
flips in skew lines, the translation acts  
along the line, which is the common perpendicular to the skew lines, 
by the distance equal twice the distance between the skew lines and 
a rotation about the same common perpendicular by the angle which is 
twice the angle between the lines. 

\begin{thm}\label{th3}
Two biflippers in the Euclidean 3-space that consist of two skew lines are
equivalent iff the common perpendicular to the skew lines in one of the
biflippers coincides with the common perpendicular to the skew lines in the
other biflipper and the biflippers can be obtained from each other by a
translation along the common perpendicular followed by a rotation about it. 
\end{thm}

\begin{proof}
Let $\overrightarrow{AB}$ is a biflipper formed of skew lines $A$ and $B$
with common perpendicular $L$.  
Translations along  $L$ and rotations about $L$ generate a
commutative group.  The screw motions encoded by
 $\overrightarrow{AB}$ belongs to this group.   
Therefore, by Theorem \ref{th02} translations along $L$
and rotations about $L$ map  $\overrightarrow{AB}$  to equivalent
biflippers. 

By the way, this centralizer group for the screw motion can be
obtained also from the decomposition of the screw motion into a direct
product of a plane rotation and a line translation: the centralizer of a
plane rotation is the group of all plane rotations with the same center,
the centralizer of a line translation is the group of all line
translations, the centralizer of their direct product is the product of the
centralizers.     

Let  $\overrightarrow{CD}$ be a biflipper equivalent to
$\overrightarrow{AB}$. By the classification of biflippers in the 3-space
made above, $C$ and $D$ are skew lines. Only one line is invariant under a 
screw motion, and the common perpendicular of  two skew lines is
invariant under the reflections in the lines, and hence under their
composition. Thus, $C$ and $D$ have the same common perpendicular as $A$
and $B$. By an appropriate composition $T$ of a translation along this line 
and rotation along it, we can map $A$ to $C$. Then $\overrightarrow{CT(B)}$
is equivalent to $\overrightarrow{AB}$ and hence to $\overrightarrow{CD}$.
So, $F_{T(B)}\circ F_C=F_D\circ F_C$. By multiplying this equality by $F_C$
from the right, we get $F_{T(B)}=F_D$ and hence $T(B)=D$.
\end{proof}

\subsection{Head to tail rule for screw motions.}\label{s7.6} 
Take biflippers $\overrightarrow{AB}$ and $\overrightarrow{CD}$  
for the screw motions, and in each of them  
extend the arrow connecting the flippers to the axes of the screw
motions   
(the common perpendicular $X$ for lines $A$ and $B$ and the common
perpendicular $Y$ for $C$ and $D$).  
Find the common perpendicular $Z$ to $X$ and $Y$. 
By gliding the biflippers along their axes and rotating about the axes,
make the head flipper of the first biflipper and the tail
flipper of the second biflipper coinciding with $Z$. 
Find the common perpendicular to the tail
flipper  
of the first biflipper and head flipper of the second biflipper. 
Connect these flippers
 with an arrow along this common perpendicular. 
Erase the old arrows and their common perpendicular.\\
\centerline{\begin{picture}(0,0)%
\includegraphics{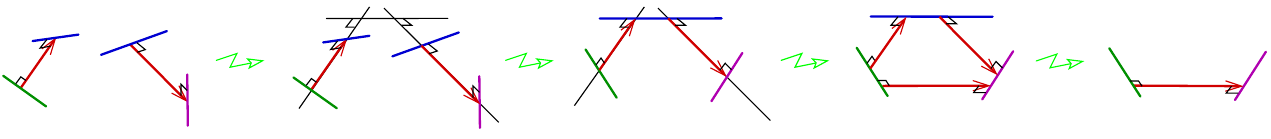}%
\end{picture}%
\setlength{\unitlength}{2901sp}%
\begingroup\makeatletter\ifx\SetFigFont\undefined%
\gdef\SetFigFont#1#2#3#4#5{%
  \reset@font\fontsize{#1}{#2pt}%
  \fontfamily{#3}\fontseries{#4}\fontshape{#5}%
  \selectfont}%
\fi\endgroup%
\begin{picture}(8292,886)(-1413,-305)
\put(-719,344){\makebox(0,0)[lb]{\smash{{\SetFigFont{8}{9.6}{\rmdefault}{\mddefault}{\updefault}{\color[rgb]{0,0,.82}$C$}%
}}}}
\put(-404,-241){\makebox(0,0)[lb]{\smash{{\SetFigFont{8}{9.6}{\rmdefault}{\mddefault}{\updefault}{\color[rgb]{.69,0,.69}$D$}%
}}}}
\put(-1079,-151){\makebox(0,0)[lb]{\smash{{\SetFigFont{8}{9.6}{\rmdefault}{\mddefault}{\updefault}{\color[rgb]{0,.56,0}$A$}%
}}}}
\put(361,-196){\makebox(0,0)[lb]{\smash{{\SetFigFont{8}{9.6}{\rmdefault}{\mddefault}{\updefault}{\color[rgb]{0,0,0}$X$}%
}}}}
\put(1801,-151){\makebox(0,0)[lb]{\smash{{\SetFigFont{8}{9.6}{\rmdefault}{\mddefault}{\updefault}{\color[rgb]{0,0,0}$Y$}%
}}}}
\put(2161,-151){\makebox(0,0)[lb]{\smash{{\SetFigFont{8}{9.6}{\rmdefault}{\mddefault}{\updefault}{\color[rgb]{0,0,0}$X$}%
}}}}
\put(3556,-151){\makebox(0,0)[lb]{\smash{{\SetFigFont{8}{9.6}{\rmdefault}{\mddefault}{\updefault}{\color[rgb]{0,0,0}$Y$}%
}}}}
\put(-1304,344){\makebox(0,0)[lb]{\smash{{\SetFigFont{8}{9.6}{\rmdefault}{\mddefault}{\updefault}{\color[rgb]{0,0,.82}$B$}%
}}}}
\put(1531,434){\makebox(0,0)[lb]{\smash{{\SetFigFont{8}{9.6}{\rmdefault}{\mddefault}{\updefault}{\color[rgb]{0,0,0}$Z$}%
}}}}
\put(3331,434){\makebox(0,0)[lb]{\smash{{\SetFigFont{8}{9.6}{\rmdefault}{\mddefault}{\updefault}{\color[rgb]{0,0,.82}$Z$}%
}}}}
\end{picture}%
}

\subsection{Rotary reflections.}\label{s7.6'} 
As we saw above, an ordered pair formed of transversal line and plane is 
a biflipper for a rotary reflection. 
Recall that a rotary reflection is a composition of a reflection in a plane 
with a rotation about a line which is perpendicular to the plane. 
The rotation and reflection commute, therefore the order does not matter.\\
\centerline{ 
$\vcenter{\hbox{\begin{picture}(0,0)%
\includegraphics{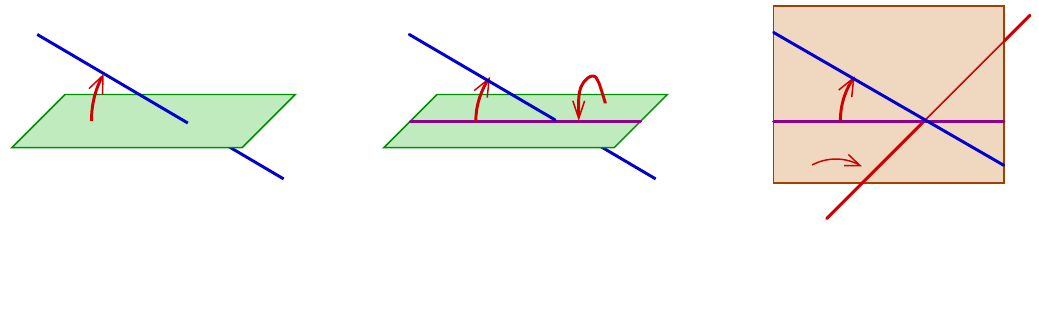}%
\end{picture}%
\setlength{\unitlength}{3729sp}%
\begingroup\makeatletter\ifx\SetFigFont\undefined%
\gdef\SetFigFont#1#2#3#4#5{%
  \reset@font\fontsize{#1}{#2pt}%
  \fontfamily{#3}\fontseries{#4}\fontshape{#5}%
  \selectfont}%
\fi\endgroup%
\begin{picture}(5257,1570)(211,-899)
\put(2194, 98){\makebox(0,0)[lb]{\smash{{\SetFigFont{11}{13.2}{\rmdefault}{\mddefault}{\itdefault}{\color[rgb]{.56,0,.56}$C$}%
}}}}
\put(2488,425){\makebox(0,0)[lb]{\smash{{\SetFigFont{11}{13.2}{\rmdefault}{\mddefault}{\itdefault}{\color[rgb]{0,0,.82}$B$}%
}}}}
\put(3292,263){\makebox(0,0)[lb]{\smash{{\SetFigFont{11}{13.2}{\rmdefault}{\mddefault}{\itdefault}{\color[rgb]{0,0,.82}$A$}%
}}}}
\put(1396,254){\makebox(0,0)[lb]{\smash{{\SetFigFont{11}{13.2}{\rmdefault}{\mddefault}{\itdefault}{\color[rgb]{0,0,.82}$A$}%
}}}}
\put(586,434){\makebox(0,0)[lb]{\smash{{\SetFigFont{11}{13.2}{\rmdefault}{\mddefault}{\itdefault}{\color[rgb]{0,0,.82}$B$}%
}}}}
\put(4141,119){\makebox(0,0)[lb]{\smash{{\SetFigFont{11}{13.2}{\rmdefault}{\mddefault}{\itdefault}{\color[rgb]{.56,0,.56}$C$}%
}}}}
\put(4366,434){\makebox(0,0)[lb]{\smash{{\SetFigFont{11}{13.2}{\rmdefault}{\mddefault}{\itdefault}{\color[rgb]{0,0,.82}$B$}%
}}}}
\put(676,-826){\makebox(0,0)[lb]{\smash{{\SetFigFont{11}{13.2}{\rmdefault}{\mddefault}{\itdefault}{\color[rgb]{0,0,0}rotary reflection = rotation $\circ$ reflection in plane}%
}}}}
\put(4771,434){\makebox(0,0)[lb]{\smash{{\SetFigFont{11}{13.2}{\rmdefault}{\mddefault}{\itdefault}{\color[rgb]{.5,.17,0}$D$}%
}}}}
\put(226,-601){\makebox(0,0)[lb]{\smash{{\SetFigFont{11}{13.2}{\rmdefault}{\mddefault}{\itdefault}{\color[rgb]{0,0,0}$F_B\circ F_A=(F_B\circ F_C)\circ(F_C\circ F_A)=F_B\circ F_C\circ F_D$}%
}}}}
\end{picture}%
}}$}

As we saw above, a rotary reflection
is a direct product of a planar rotation and reflection of a line. 
The centralizer group of a rotary reflection is the product of the
centralizers of the factors.  Therefore the centralizer of a
rotary reflection is the direct product of the group of plane rotations about
the fixed point of the rotary reflection and the reflection of the 3-space in
the plane of rotations. 

A rotary reflection has a fixed point and invariant plane and line orthogonal
to each other and intersecting in the fixed point. Flippers in a biflipper,
which encodes a rotary reflection are a line and a plane, the plane contains
the invariant line and the line is contained in the invariant plane and
passes through the fixed point. 

Using these one can prove that two biflippers of rotary reflections are
equivalent iff they can be obtained from each other by a rotation about the
invariant line and reflection in the plane flipper of the biflipper with a
simultaneous change of the flippers' ordering.  

\subsection{No head to tail for rotary reflections.}\label{s7.7}
If two rotary reflections are such that their invariant lines are skew and the 
common perpendicular to the  lines does not pass through the fixed
points of both rotary reflections, then there are no biflippers of the
rotary reflections with a common flipper. Thus, there is no head to tail
method
that works for composing any two rotary reflections in the 3-space.

\subsection{Screw motions help.}\label{s7.8}
However, one can easily calculate composition $S\circ T$ of rotary reflections
using a head to tail rule for composing two screw motions. For this, first
choose biflippers $\overrightarrow{AB}$ and $\overrightarrow{CD}$ for $T$
and $S$, respectively, such that $B$ and $C$ are planes. Then consider the
biflipper $\overrightarrow{BC}$. Since it consists of planes, it encodes
either translation (if $B\parallel C$) or rotation (otherwise). Take any
plane $E$ orthogonal to both $B$ and $C$. By Theorem \ref{th01}, the
biflipper $\overrightarrow{B'C'}$ formed of the lines
$B'=B\cap E$ and $C'=C\cap E$ is equivalent to $\overrightarrow{BC}$. Thus
$S\circ T=F_A\circ F_{B'}\circ F_{C'}\circ F_D$, and we deal with the
composition of screw motions or their degenerations, and can apply
the technique of Section \ref{s7.5}.

\subsection{No head to tail for composition of rotary reflection and screw
motion.}\label{s7.9}
Similarly, if a rotary reflection and screw motion are such that their
invariant lines are skew and the common perpendicular to the lines does not
pass through the fixed point of the rotary reflection, then the
rortoreflection and screw motion have no biflippers with a common
mirror, and there is no head to tail rule
that would work for composing them. 

\subsection{Screw motions help again.}\label{s7.10}
However, the composition can be easily calculated. For this one can use
again a decomposition of a rotary reflection into a composition of a screw
motion and reflection. 

Say, if $T$ is a
rotary reflection, and $S$ is a screw motion, then choose biflippers  
$\overrightarrow{AB}$ and $\overrightarrow{CD}$ for $T$
and $S$, respectively, such that $A$ is a plane. Let $E$ be any line on
$A$ and $F$ be a plane containing $E$ and perpendicular to $A$. Then
$F_A=F_E\circ F_F$ and $S\circ T=F_D\circ F_C\circ F_B\circ F_A=F_D\circ
F_C\circ F_B\circ F_E\circ F_F$. For $F_D\circ
F_C\circ F_B\circ F_E$ one can apply a head to tail composition method and
replace it by $F_X\circ F_Y$. Now we have to reduce the composition
$F_X\circ F_Y\circ F_F$, where $X$ and $Y$ are lines and $F$ is a plane. 
Denote the common perpendicular of $X$ and $Y$ by $Z$. 

If $F$ is not parallel to $Z$, then, 
by a translation along $Z$, we move the point $Z\cap Y$ to $F$ and, by a
rotation about $Z$ place $Y$ on $F$. As $Y\subset F$, the composition
$F_Y\circ F_F$ is a reflection in a plane $W$ containing $Y$ and perpendicular
to $F$. Hence $S\circ T=F_X\circ F_W$ and we are done. 

If $F$ is parallel to $Z$, then by a rotation about $Z$ we make $Y$
perpendicular to $F$, and observe that, as soon as this is done, 
$F_Y\circ F_F=F_{Y\cap F}$. Hence, 
$S\circ T=F_X\circ F_Y\circ F_F=F_X\circ F_{Y\cap F}$.

\section{In hyperbolic 3-space}\label{s8}

\subsection{Well-known facts.}\label{s8.0}
We will use the Poincar\'e model, in which the hyperbolic 3-space is
represented by an open unit ball $H^3$ in $\R^3$.
The boundary 2-sphere $\p H^3$ of $H^3$ is called the {\sfit absolute}.

In this model, a plane is either an open unit disk, which is cut on $H^3$  
by a plane passing through the origin, or a surface, which cut on $H^3$
by a 2-sphere orthogonal to $\p H^3$. In any case the boundary of a plane in
$\R^3$ is a circle which lies on $\p H^3$. It is called the {\sfit absolute\/} 
of this plane.
A line in $H^3$ is an arc which can be presented as the intersection of two
different planes. A line is either a diameter of the ball $H^3$, or an arc
cut on $H^3$ by a circle orthogonal to $\p H^3$.

Two planes may intersect in a line or be disjoint. In the latter case,
their absolutes may be disjoint and then the planes are said to be {\sfit
ultra-parallel\/} or be tangent to each other and then the planes are said
to be {\sfit parallel}. 
Two ultra-parallel planes have a unique common perpendicular line. Conversely, 
two planes perpendicular to the same line are ultra-parallel.
 
There are three kinds of flippers: planes, lines and points. 
Any plane, line or point is a flipper.
In the Poincar\'e model, the flip in a plane, which passes through the
origin of $\R^3$ is the restriction of the usual Euclidean 
symmetry of $\R^3$ with respect to the 2-subspace, which cut the plane on
$H^3$; the flip in a plane, 
which is cut by a 2-sphere $S$, is a restriction  of the inversion of $\R^3$ 
with respect to $S$. A flip in a line is a composition of flips in any two 
planes passing through the line and orthogonal to each other. Similarly,
the flip in a point is the composition of flips in any three planes passing
through this point and orthogonal to each other. 

\subsection{Biflippers.}\label{s8.1} 
We  restrict ourselves to generic biflippers leaving to the reader to
compile a complete list of biflippers in the hyperbolic 3-space. As in
other situations that we considered above, we just list ordered pairs of
flippers and figure out what the compositions of the corresponding flips
are. See \cite{Fenchel}, chapter IV.
\medskip 

\centerline{$\sfig{3D-hyp-frAr-1}{.9}$}
\centerline{$\fig{3D-hyp-frAr-2}$}
\medskip

\noindent
{\bf Exercise.} Find an explicit descriptions for equivalence of the biflippers 
and head to tail rules.
Advice: revisit sections \ref{s3} and \ref{s7}. 

\section{Full M\"obius group.}\label{s9}
An isometry of the hyperbolic 3-space $H^3$ acts on the absolute $\p H^3$.
Recall that the absolute $\p H^3$ is a 2-sphere. 
In the Poincar\'e model, $\p H^3$ is the unit sphere 
$S^2=\{x\in\R^3\mid |x|=1\}$ in $\R^3$. Isometries of $H^3$ induce on it
transformations which preserve angles, but do not necessarily preserve 
any metric or orientation. They form a group, which is called the {\sfit full
M\"obius group\/}. As an abstract group, it is isomorphic to the group of
all isometries of the hyperbolic 3-space. The transition from an isometry
of $H^3$ to the corresponding transformation of $\p H^3$ is an isomorphism
between the groups. 

The M\"obius group is a subgroup of the full M\"obius group formed by
orientation preserving transformations. A general trend to discriminate
orientation reversing maps and non-orientable manifolds in the
2-dimensional case is rationalized by an attention to complex structures. 
Elements of the M\"obius group are identified with fractional linear
transformations with complex coefficients, the 2-sphere is idetified with  
the complex projective line $\C P^1$ and the M\"obius group with the group
of its complex projective automorphisms. 

In our context, orientation reversing elements of the full M\"obius
group are not any worse than orientation preserving ones. The group of 
isometries of $H^3$ includes orientation reversing isometries. The group is
generated by reflections in planes. Respectively, the full M\"obius group 
is generated by involutions, which fixed point sets are
circles (the absolutes of the planes). 
If the fixed point set is a great circle, then the involution of $S^2$ is
the restriction of the reflection of $\R^3$ in the plane, which cut the
circle on the sphere. If the fixed point set is not a great circle, it is
cut on $S^2$ by a sphere $S$ orthogonal to $S^2$, and the involution is the
restriction to $S^2$ of the inversion of $\R^3$ in $S$. 

As was mentioned in the beginning of the paper, the notions of flip, 
flipper and biflipper are naturally extended to this setup. 
There are two kinds of flippers: circles that are cut on $S^2$ by planes
and pairs of points. The flips corresponding to circles are the involutions
which were described above as the generators of the M\"obius group. The flips
corresponding to pairs of points are compositions of two commuting flips of
the first kind.

In the M\"obius group there are many involutions without fixed points. They
are induced by flips on $H^3$ with one point flippers. These
involutions of $S^2$ share the fixed point set in $S^2$ (the empty), hence
they are not determined by their fixed point sets and are not qualified 
to be flips in $S^2$.

The stereographic projection $st:S^2\sminus pt\to\R^2$ is a conformal
isomorphism, and we can use it for pictures. 
Here is the list of plane images of the biflippers. The first row is filled
by biflippers of M\"obius transformations. 
The biflippers of the second row
correspond to orientation reversing transformations.
\smallskip

\centerline{$\fig{moeb-frAr-1}$}
\centerline{$\fig{moeb-frAr-2}$}

\medskip

\noindent{\bf Conclusion.} Presentations of isometries as compositions of two
flips allow one to find {\sfit geometric\/} algorithms for
calculating compositions. In the two-dimensional classical homogeneous spaces
this works perfectly. 
The two-dimensional geometry is the most interesting
for the pedagogical applications.

As the dimension grows the algorithms become more complicated. In the
three-dimensional Euclidean space the algorithms are still easy, but for
some pairs of isometries,  flip-flop decompositions of them such that the same 
flipper appears in both decompositions do not exist. 


\end{document}